\documentclass[fleqn,a4paper,10pt]{amsart}

\usepackage{graphicx}
\usepackage{tikz}
\usetikzlibrary{shapes,decorations.pathmorphing}
\usepackage{amsmath}
\usepackage{amssymb,amsthm}
\usepackage{comment}
\usepackage{hyperref}

\def\cM{\mathcal{M}}
\def\cQ{\mathcal{Q}}

\def\cMh{\widehat{\cM}}
\def\cQh{\widehat{\cQ}}

\newtheorem{theo}{Theorem}
\newtheorem{prop}{Proposition}
\newtheorem{lem}{Lemma}

\theoremstyle{definition}
\newtheorem{rmk}{Remark}

\title{The three-point function of general planar maps}
\author{\'Eric Fusy}
\address{LIX, \'Ecole Polytechnique, 91128 Palaiseau, France}
\email{fusy@lix.polytechnique.fr}

\author{Emmanuel Guitter}
\address{Institut de Physique Th\'eorique, CEA, IPhT, 91191 Gif-sur-Yvette, France, CNRS, URA 2306}
\email{emmanuel.guitter@cea.fr}

\begin{document}
\maketitle

\begin{abstract}

We compute the distance-dependent three-point function of general planar maps and of bipartite planar maps, i.e., the generating
function of these maps with three marked vertices at prescribed pairwise distances. Explicit expressions are given for
maps counted by their number of edges only, or by both their numbers of edges and faces. A few limiting cases and 
applications are discussed.

\end{abstract}

\section{Introduction}
\label{sec:introduction}

The statistics of distances in planar maps is a subject of constant interest, whose study
already produced a lot of remarkable results but still leaves many open questions. 
A key ingredient in this study was the discovery by Schaeffer \cite{SchPhD,CMS09} (giving a  reformulation of a bijection due to Cori and Vauquelin~\cite{CoriVa}) 
of a distance-preserving bijective coding of
planar quadrangulations (later generalized to planar maps with arbitrary face degrees \cite{BDG04}) by decorated trees, reducing 
de facto a number of distance statistics problems to the simpler question of enumerating these trees. 
This approach was used to obtain the so-called {\it two-point function} which, by enumerating maps with two marked 
vertices at a prescribed distance, provides a measure of the statistics of distances between two random points in the map. 
In a first stage, the explicit expression of the two-point function was derived in~\cite{GEOD,BG12} for a number of families of maps with controlled {\it bounded} 
face degrees (triangulations, quadrangulations, ...), corroborating the original prediction of Ambj\o rn and Watabiki \cite{AmWa95}. 
The restriction to maps with bounded face degrees was lifted in a second, very recent, stage thanks to the discovery by
Ambj\o rn and Budd \cite{AmBudd} of yet another distance-preserving bijective coding of planar quadrangulations, now by general 
planar maps with {\it unbounded} face degrees. This latter coding, combined with the original Schaeffer bijection, 
allowed in turn to get an explicit expression for the two-point function of general planar maps controlled by their number of edges only
or by both their numbers of edges and faces.
This approach was then generalized in \cite{BFG} to obtain the two-point function of other families of maps and hypermaps 
with unbounded face degrees, controlled by their number of edges or their numbers of hyperedges and faces.
 
The more involved question of the {\it three-point function}, now enumerating maps with three marked 
vertices at prescribed pairwise distances, hence providing a refined information on the correlations
between mutual distances, was so far solved only in the simplest case of planar quadrangulations \cite{BG08}.
The solution relies on a natural extension by Miermont \cite{Miermont2009} of the original Schaeffer bijection and constituted 
so far the most advanced result on the statistics of distances within maps at the discrete level. 

In this paper, we show how to take advantage of the new Ambj\o rn-Budd bijection to get an explicit
expression for {\it the three-point function of general planar maps controlled by both their numbers of edges and faces}.
We also obtain the three-point function of the subclass of general planar maps made of 
{\it bipartite planar maps}.

The paper is organized as follows: Section 2 explains how to use the Ambj\o rn-Budd bijection to obtain
a bijective coding of bi-pointed and tri-pointed planar maps, i.e., maps with respectively two and 
three marked vertices, which preserves the information on the distance between the marked vertices.
The coding involves what we call $(s,t)$- and $(s,t,u)$-well-labelled maps, which are
maps with respectively two and three faces, whose vertices carry labels subject to a number of constraints
involving the parameters $s$, $t$, and $u$, themselves directly related to the prescribed pairwise 
distances between the marked vertices. The case of bipartite planar maps is obtained by keeping only
a subclass of these well-labelled maps, which we call very-well-labelled. 
We then exploit this coding in Section 3 to derive explicit expressions for the two-point (Sect.~3.1)
and three-point (Sect.~3.2) functions of general planar maps controlled by their number of edges
only. The case of bipartite planar maps is addressed in Sect.~3.3.     
Our derivation relies on the preliminary knowledge of the generating functions of
a number of well-labelled or very-well-labelled objects, an information which follows either
from previously known expressions or from the resolution of new recursion relations.
We then repeat our analysis in Section 4 to obtain {\it bivariate} two- and three-point functions,
with a control on both the number of edges and the number of faces, in the case of general 
maps (Sect.~4.1) or of bipartite maps (Sect.~4.2). We also discuss at this level the limit of
tri-pointed maps having a minimal number of faces (in practice one or two faces).
Section 5 is devoted to a number of applications of our results in the scaling limit
of maps with a large number of edges (and a fixed weight per face). This includes
the average number of vertices or edges lying on a geodesic path between two given 
far away vertices.
We gather our conclusions in Section 6.

\section{Bijective coding of bi-pointed and tri-pointed maps}
\subsection{Definitions}
A \emph{map} denotes a connected graph embedded on the sphere. A \emph{bipartite map} is
a map with all faces of even degree, and a \emph{quadrangulation} is a map with all
faces of degree $4$. A well-labelled map (resp. very-well-labelled map) 
is a map $M$ together with an assignment: $\ell:V\to\mathbf{Z}$ 
of integers to its vertices, such that for each edge $e=\{u,v\}$ of $M$, $|\ell(v)-\ell(u)|\leq 1$ (resp. 
$|\ell(v)-\ell(u)|= 1$). A very-well-labelled map
is necessarily bipartite. A vertex of a well-labelled map is called a \emph{local max} (resp.\ \emph{local min}) if it is not smaller (resp.\ not larger) than any of its neighbours. For a face $f$ of a well-labelled
map, denote by $\min(f)$ (resp.\ $\max(f)$) the minimum (resp.\ maximum) label over 
 vertices incident to $f$. Denote by $\cMh$ the set of well-labelled maps and by $\cQh$
the set of very-well-labelled quadrangulations. 

\subsection{The Ambj\o rn-Budd bijection}
Let $\Phi$ be the mapping from $\cQh$ to $\cMh$ such that, for $q\in\cQh$, $\Phi(q)$ is the map
 obtained from $q$ by applying the local rules of Figure~\ref{fig:bij_ab_loc} 
(left-part) to each face of $q$, then
deleting all local max and all edges of $q$. Similarly let $\Phi^-$ be the mapping from 
$\cQh$ to $\cMh$ such that, for $q\in\cQh$, $\Phi^-(q)$ is the map 
 obtained from $q$ by applying the local rules of Figure~\ref{fig:bij_ab_loc} 
(right-part) to each face of $q$, then
deleting all local mins and all edges of $q$. Ambj\o rn and Budd showed \cite{AmBudd} 
that $\Phi$ and
$\Phi^-$ are bijections between $\cQh$ and $\cMh$ that preserve several parameters, from which it follows (see Figure~\ref{fig:bijab} for an example):

\begin{figure}
\begin{center}
\includegraphics[width=12cm]{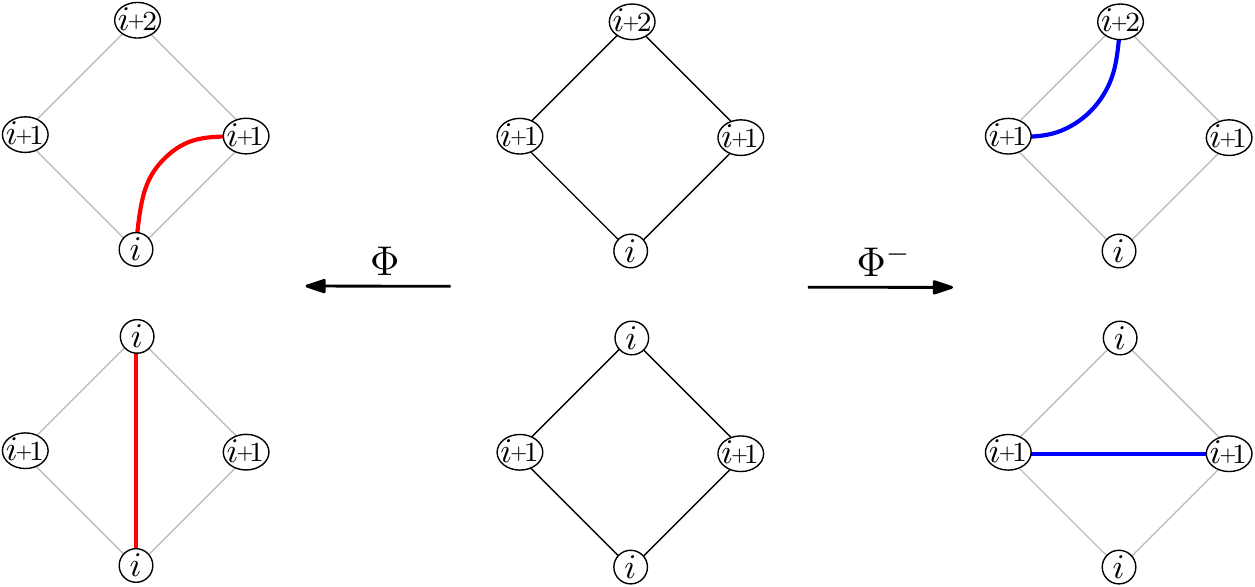}
\end{center}
\caption{Local rules of the bijections $\Phi$ (left) and $\Phi^-$ (right) applied within each
face of a very-well-labelled quadrangulation. The rules on the right are those of the Schaeffer
bijection \cite{SchPhD,CMS09}}
\label{fig:bij_ab_loc}
\end{figure}

\begin{figure}
\begin{center}
\includegraphics[width=13cm]{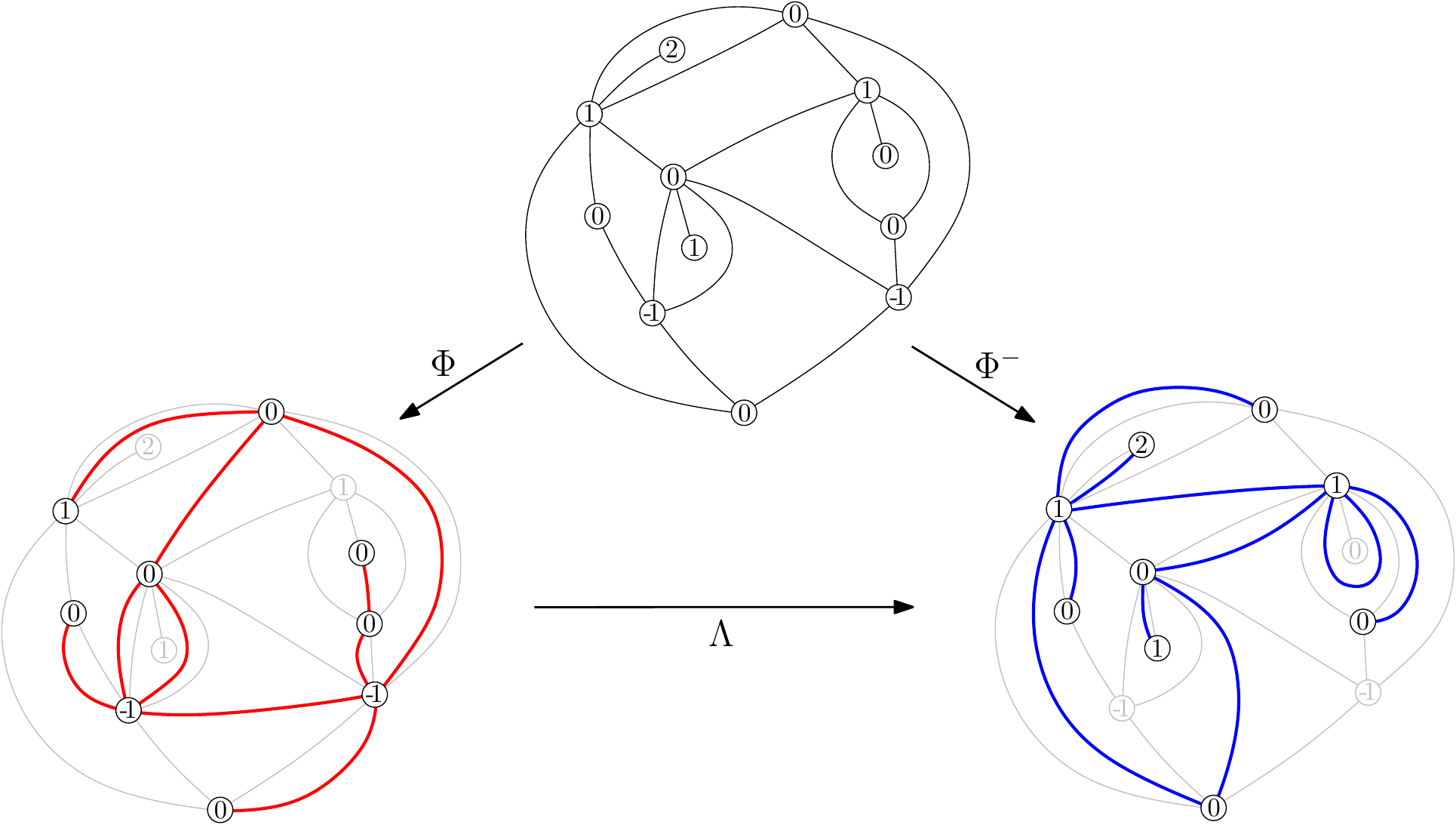}
\end{center}
\caption{The Ambj\o rn-Budd bijection on an example: each face $f'$ of $M'=\Lambda(M)$ corresponds
to a local min $v$ of $M$ (with $\ell(v)=\min(f')-1$, and with $v$ 
within $f'$ when $M$ and $M'$ are superimposed),  
and each face $f$ of $M$ corresponds to a local max $v'$ of $M'$ (with $\ell(v')=\max(f)+1$, and  
with $v'$ within $f$ when $M$ and $M'$ are superimposed).}
\label{fig:bijab}
\end{figure}

\begin{theo}[Ambj\o rn-Budd]\label{theo:ambu}
The mapping $\Lambda=\Phi^-\circ \Phi^{-1}$ is a bijection between $\cMh$ and $\cMh$
that preserves the number of edges, such that for $M\in\cMh$ and $M'=\Lambda(M)$, 
each local min $v$ of $M$ corresponds to a face $f'$ of $M'$, with $\ell(v)=\min(f')-1$, 
and each face $f$ of $M$ corresponds to a local max $v'$ of $M'$, with $\max(f)+1=\ell(v')$. 
\end{theo}

\begin{rmk}
Note that  each edge $e'$ of labels $i-i$ in $M'=\Lambda(M)$ is ``dual''
to an edge of labels $(i-1)-(i-1)$ in $M$, such that (when $M$ and $M'$ are superimposed),  
$e'$ crosses $e$ (within the same face of the associated $q\in\cQh$). 
This also ensures that $\Lambda$ and $\Lambda^{-1}$ preserve the property of being very-well-labelled. 
\end{rmk}

We now state a useful lemma (in view of proving bijections for bi-pointed and tri-pointed
maps later on).

\begin{figure}
\begin{center}
\includegraphics[width=13cm]{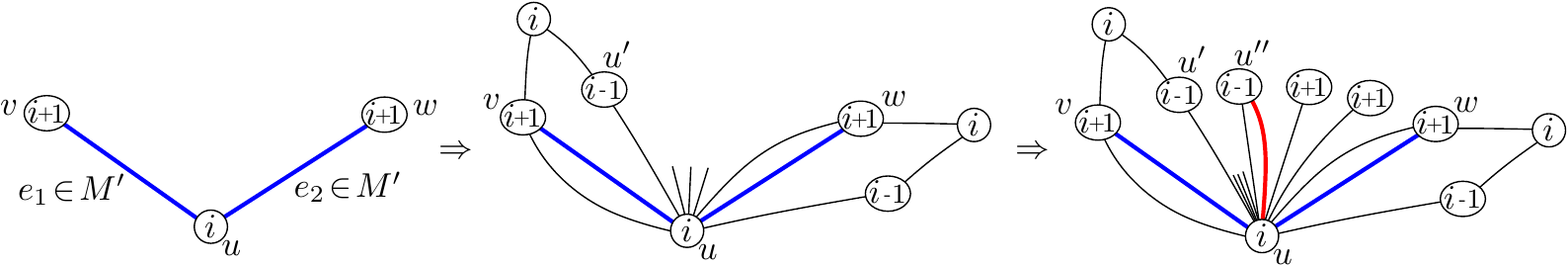}
\end{center}
\caption{The situation in the proof of Lemma~\ref{lem:neighbour}.}
\label{fig:proof_lem}
\end{figure}

\begin{lem}\label{lem:neighbour}
Let $M\in\cMh$ and $M'=\Lambda(M)$, considered as superimposed (via the associated $q\in\cQh$). 
Consider two edges $e_1,e_2$ of $M'$ incident to a vertex $u$ of $M'$, of label $i$,
 such that the extremities
$v$ and $w$ of $e_1$ and $e_2$ have label $i+1$. Let $S$ be the clockwise angular
sector between $e_1$ and $e_2$ around $u$. Then there is an edge 
of $M$ leaving $u$ in the sector $S$ and ending at a vertex of label $i-1$. 
\end{lem}
\begin{proof}
The situation is shown in Figure~\ref{fig:proof_lem}. Let $q$ be the very-well-labelled quadrangulation
associated to $M$ (via $\Phi$) and to $M'$ (via $\Phi^-$). The local rules of $\Phi^-$ ensure
that $q$ has an edge from $u$ to $v$ (just to the left of $e_1$), an edge from $u$ to $w$ (just to the left of $e_2$),
and that the next edge of $q$ after $e_1$ in clockwise order around $u$ leads to a vertex of label $i-1$; 
denote by $u'$ this vertex. Starting from $u'$, let $u''$ be the last neighbour of $u$ of label $i-1$ (in clockwise order 
around $u$) before $w$. Then the local rules of $\Phi$ ensure that there is an edge of $M$ from $u$ to $u''$, and this edge is in $S$. 
\end{proof}

\subsection{Bijection for bi-pointed maps}
A \emph{bi-pointed map} is a map with two (different and distinguished) marked vertices, denoted $v_1$ and $v_2$. 
The distance between $v_1$ and $v_2$ is called the \emph{two-point distance} of $M$
and is denoted $d_{12}(M)$. For $s,t$ two positive integers, define an \emph{$[s,t]$-well-labelled map} 
as a bi-pointed well-labelled map such that $v_1$ and $v_2$ are the only local min, and
have respective labels $-s$ and $-t$. Define an \emph{$(s,t)$-well-labelled map} as a
well-labelled map with exactly two faces (which are distinguished) $f_1$ and $f_2$, such that
$\min(f_1)=-s+1$ and $\min(f_2)=-t+1$. In such a well-labelled map, a \emph{border-vertex} (resp.\ \emph{border-edge}) 
is a vertex (resp.\ an edge) incident to $f_1$ and $f_2$. 
An $(s,t)$-well-labelled map is said to be \emph{of type A} if
 the minimum label over all \emph{border-vertices}  is $0$, and none of the \emph{border-edges} 
has labels $0-0$. An $(s,t)$-well-labelled map is said to be \emph{of type B} if
 the minimum label over all border-vertices is $0$, and there is at least one
 border-edge with labels $0-0$. 

\begin{prop}\label{prop:bij_bipointed}
For $s,t\geq 1$, the mapping $\Lambda$ induces a bijection between $[s,t]$-well-labelled maps and $(s,t)$-well-labelled maps. For $M$ an $[s,t]$-well-labelled map and $M'=\Lambda(M)$, one has
$d_{12}(M)=s+t$ iff $M'$ is of type $A$, and $d_{12}(M)=s+t-1$ iff $M'$ is of type $B$.  
\end{prop}
\begin{proof}
We just have to show the statements of the second sentence 
(the first sentence immediately follows from Theorem~\ref{theo:ambu}). 
Let $\Gamma$ be the submap of $M'$ formed by the border-vertices and border-edges (note that $\Gamma$ is a cycle). Denote by $a$ the minimum label over all vertices of $\Gamma$. Let 
$P$ be a path in $M$ from $v_1$ to $v_2$, and let $p$ be the
first intersection of $P$ with $\Gamma$. Two cases can occur: 
if $p$ is at a vertex of $\Gamma$, of label $i$, then
the portion of $P$ before $p$ (resp.\ after $p$) has length at least $s+i$ (resp.\ at least $t+i$) because the label-increment along each edge is at most $1$; hence $P$ has length at least $s+t+2i$. 
If $p$ is at the middle of an edge $e'$ of $\Gamma$, of labels $i-i$, then $e'$ is ``dual'' to an edge $e$
of $M$ of labels $(i-1)-(i-1)$, 
and $P$ hits $p$ when traversing $e$. 
Again, the portion of $P$ before $e$ (resp.\ after $e$) has length at least $s+i-1$ (resp.\
at least $t+i-1$), so that $P$ has length at least $s+t+2i-1$. This observation ensures that if 
 $d_{12}(M)=s+t$ then the labels $i$ on $\Gamma$ cannot all be positive, hence $a\leq 0$,  and if $d_{12}(M)=s+t-1$ then there
 exists on $\Gamma$ a vertex with label $i<0$ or an edge with labels $0-0$, hence either we have $a< 0$ or we have $a=0$
 and there is at least one border-edge of labels $0-0$.  

We now aim at obtaining an upper bound for $d_{12}(M)$ in terms of $a$. Observe that 
two cases can arise: (i) no border-edge has labels $a-a$, (ii) at least one border-edge has 
labels $a-a$. In case (i), let $v$ be a border-vertex of label $a$. Its neighbours on $\Gamma$ have
label $a+1$, so by Lemma~\ref{lem:neighbour}, $v$ has (in $M$) a neighbour $w$ of label $a-1$ that belongs
to $f_1$ (note that $w$ is strictly in $f_1$ since all border-vertices have label at least $a$). 
 Let $v_1$ be the local min of $M$ in $f_1$. If $w\neq v_1$, then $w$ is not a local min of $M$, 
so that $w$ has (in $M$) a neighbour of label $a-2$, strictly in $f_1$. 
Continuing this way one builds a label-decreasing
path (staying strictly in $f_1$) from $v$ to $v_1$, hence of length $s+a$. Similarly one can build a  label-decreasing
path from $v$ to $v_2$, of length $t+a$. Hence $d_{12}(M)\leq s+t+2a$. 
In case (ii), let $e'$ be a border-edge of labels $a-a$, and let $e\in M$ be its dual edge, with labels $(a-1)-(a-1)$. Let $w_1$ be the extremity of $e$ in $f_1$ and $w_2$ the extremity of $e$ in $f_2$ 
(note that $w_1$ is strictly in $f_1$ and $w_2$ is strictly in $f_2$).
Again, if $w_1\neq v_1$, then $w_1$ has (in $M$) a neighbour of label $a-2$, and continuing this way one
builds a label-decreasing path $P_1$ between $w_1$ and  $v_1$, of length $s+a-1$. Similarly one can
build  a label-decreasing path $P_2$ between $w_2$ and $v_2$, of length $t+a-1$. Concatenating 
$P_1$, $e$, and $P_2$ yields a path between $v_1$ and $v_2$ of length $s+t+2a-1$.  Hence $d_{12}(M)\leq s+t+2a-1$. Using the first paragraph of the proof, 
we conclude that, if $d_{12}(M)=s+t$, then $a=0$ and we are in case (i), i.e., $M'$ is of type A;
 and if $d_{12}(M)=s+t-1$, then $a=0$ and we are in case (ii), i.e., $M'$ is of type B. 

Conversely, if $M'$ is of type A, then by the arguments of the first paragraph, $d_{12}(M)\geq s+t$; 
and by the arguments of the second paragraph,   
for each border-vertex $v$ of label $0$ one can build a path of length $s+t$ from $v_1$ to $v_2$ passing
by $v$ at distance $s$ from $v_1$. Hence $d_{12}(M)=s+t$. 
And similarly if $M'$ is of type B, then by the arguments of the first paragraph, 
$d_{12}(M)\geq s+t-1$; and by the arguments of the second paragraph,  
for each border-edge $e'$ of labels $0-0$ 
one can build a path of length $s+t-1$ from $v_1$ to $v_2$, with $e$ as the $s$th edge along the path.
Hence $d_{12}(M)=s+t-1$. 
\end{proof}

\begin{figure}
\begin{center}
\includegraphics[width=13cm]{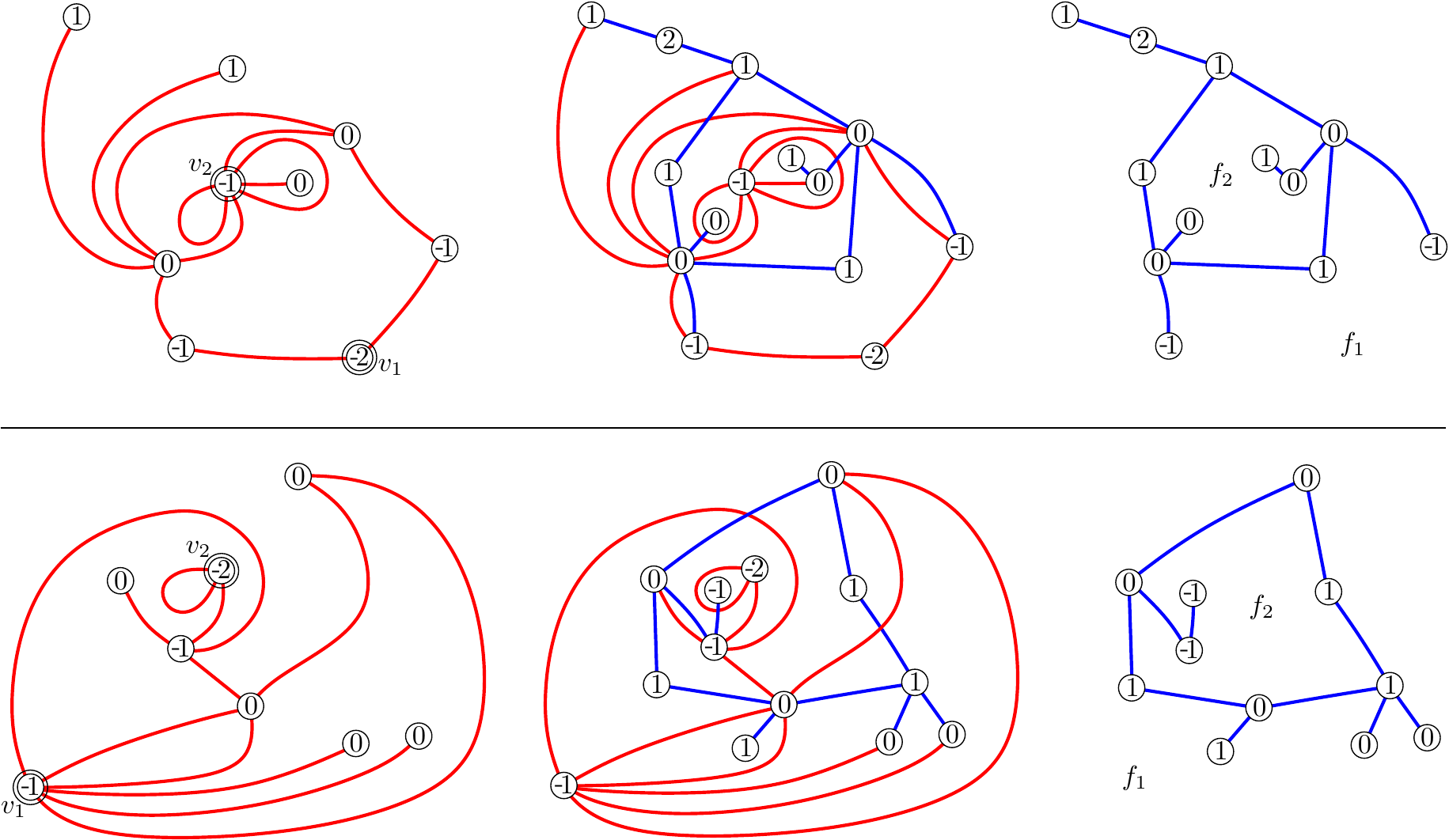}
\end{center}
\caption{Top: the left-part shows a bi-pointed map of two-point distance $3$, endowed with its $[1,2]$-well-labelling; the right-part shows the associated $(1,2)$-well-labelled map of type A (and the middle-part shows both drawings superimposed). Bottom: the left-part shows a bi-pointed map of two-point distance $2$, endowed with its $[1,2]$-well-labelling; the right-part shows the associated $(1,2)$-well-labelled map of type B (and the middle-part shows both drawings  superimposed).}
\label{fig:bij_bipointed}
\end{figure}

The proposition then yields the following bijection (see Figure~\ref{fig:bij_bipointed} for examples):

\begin{theo}\label{theo:bij_bipointed}
Let $s,t$ be two positive integers. 

(A): the mapping $\Lambda$ induces a bijection (via $[s,t]$-well-labelled maps) 
between bi-pointed maps of two-point distance $s+t$, and $(s,t)$-well-labelled maps of type $A$;
the bi-pointed map is bipartite iff the $(s,t)$-well-labelled map of type A is very-well-labelled
\footnote{Note that an $(s,t)$-well-labelled map which is very-well-labelled (for short an
$(s,t)$-very-well-labelled map) is of type A iff the minimum label over border vertices is $0$.}.

(B): the mapping $\Lambda$ induces a bijection  (via $[s,t]$-well-labelled maps) 
between bi-pointed maps of two-point distance $s+t-1$, 
and $(s,t)$-well-labelled maps of type~$B$. 

In both bijections (A) and (B), each face of the bi-pointed map corresponds to a local max in 
the associated $(s,t)$-well-labelled map.
\end{theo}
\begin{proof}
It suffices to show that a bi-pointed map $M$ of two-point distance $s+t$ (resp $s+t-1$) can be
uniquely labelled as an $[s,t]$-well-labelled map, and then to apply Proposition~\ref{prop:bij_bipointed}. 
Consider the labelling-assignment $\ell_0(v)=\mathrm{min} (d(v,v_1)-s,d(v,v_2)-t)$, where $d(u,w)$ denotes the distance (in $M$) between $u$ and $w$.  Clearly $\ell_0$ is a well-labelling and any vertex $v\notin\{v_1,v_2\}$ is not a local min. And since $|s-t|< s+t=d(v_1,v_2)$ 
(recall that $s$ and $t$ are positive), it is easy to see that $v_1$ and $v_2$ are local min, 
of respective labels $-s$ and $-t$. This proves the existence of such a labelling. 
To prove uniqueness, consider a labelling $\ell$ that makes $M$ $(s,t)$-well-labelled. For any
vertex $v$ of $M$ which is not a local min, one can find a label-decreasing path $P$ to a local min, say $v_1$. 
Then $\ell(v)-\ell(v_1)=|P|\geq d_{12}(M)$, hence $\ell(v)\geq \ell_0(v)$. 
Moreover, since the increments along edges are at most $1$, one also has (considering
increments along the geodesic paths from $v$ to $v_1$ and from $v$ to $v_2$) $\ell(v)\leq\ell_0(v)$. Hence 
$\ell(v)=\ell_0(v)$ and we have the uniqueness of such a labelling.   

Finally note that, for a bi-pointed map $M$ with $d_{12}(M)=s+t$, $\ell_0$ is a very-well labelling
iff $M$ is bipartite (indeed, if $M$ is bipartite, for $v$ a vertex of $M$, either both
$d(v,v_1)-s$ and $d(v,v_2)-t$ are even, or both $d(v,v_1)-s$ and $d(v,v_2)-t$ are odd, depending on whether $v$ is in one vertex-color or the other).  
\end{proof}

\begin{rmk}\label{rk:bi_pointed}
The arguments of the proof of Proposition~\ref{prop:bij_bipointed} 
imply that --- $M$ denoting the bi-pointed map
and $M'$ the corresponding $(s,t)$-well-labelled map --- in case (A) of Theorem~\ref{theo:bij_bipointed}, 
each border-vertex of label $0$ in $M'$ corresponds to a vertex $v$ of $M$ that belongs to a (not necessarily unique) geodesic path from $v_1$ to $v_2$, with $v$ at distance $s$ from $v_1$,
and in case (B)  of Theorem~\ref{theo:bij_bipointed}, 
each border-edge of labels $0-0$ in $M'$ corresponds to an edge $e$ of $M$ that belongs to a (not necessarily unique) geodesic path from $v_1$ to $v_2$, with $e$ the $s$th edge along the path.
\end{rmk}

\subsection{Bijection for tri-pointed maps}
A \emph{tri-pointed map} is a map $M$ with three (different and distinguished) marked vertices, denoted $v_1,v_2,v_3$. 
Denote by $d_{12},d_{13},d_{23}$ 
the distances between the marked vertices, i.e., $d_{ij}$ is the distance between $v_i$ and $v_j$.

For $s,t,u$ three positive integers, define an \emph{$[s,t,u]$-well-labelled map} 
as a tri-pointed well-labelled map such that $v_1, v_2, v_3$ are the only local min, and
have respective labels $-s,-t,-u$. Define an \emph{$(s,t,u)$-well-labelled map} as a
well-labelled map with exactly three faces (which are distinguished) $f_1,f_2,f_3$, such that
$\min(f_1)=-s+1$, $\min(f_2)=-t+1$, and $\min(f_3)=-u+1$.  For such a map,
a \emph{border-vertex} is a vertex incident to at least two different faces, and a \emph{border-edge} 
is an edge incident to two different faces. More precisely, for $1\leq i<j\leq 3$,
an \emph{$(i,j)$-border vertex} is a vertex incident to both $f_i$ and $f_j$ (and possibly also to the other face), and an \emph{$(i,j)$-border edge} is an edge incident to both $f_i$ and $f_j$. 
An $(s,t,u)$-well-labelled map is said to be \emph{of type A} if the minimum label over all
border-vertices is $0$, there is no border-edge of labels $0-0$, and 
for all $1\leq i<j\leq 3$ there is at least one $(i,j)$-border vertex of label $0$. 
An $(s,t,u)$-well-labelled map is said to be \emph{of type B} if the minimum label over all
border-vertices is $0$, and 
for all $1\leq i<j\leq 3$ there is at least one $(i,j)$-border edge of labels $0-0$. Note in particular
that, in an $(s,t,u)$-well-labelled map of type A or B, any of the three faces is adjacent to the two others.
 
\begin{prop}\label{prop:bij_tripointed}
For $s,t,u$ three positive integers, the mapping $\Lambda$ induces a bijection between 
$[s,t,u]$-well-labelled maps and $(s,t,u)$-well-labelled maps. 
For $M$ an $[s,t,u]$-well-labelled map and $M'=\Lambda(M)$, the distances $d_{12}, d_{13},d_{23}$ of $M$ 
satisfy 
$$
d_{12}=s+t,\ \ d_{13}=s+u,\ \ d_{23}=t+u
$$
iff $M'$ is of type A, and satisfy
$$
d_{12}=s+t-1,\ \ d_{13}=s+u-1,\ \ d_{23}=t+u-1
$$
iff $M'$ is of type B. 
\end{prop}
\begin{proof}
Again we just have to show the statements of the second sentence 
(the first sentence immediately follows from Theorem~\ref{theo:ambu}). 
Let $\Gamma$ be the embedded subgraph of $M'$ formed by the border-vertices and border-edges
($\Gamma$ cuts the sphere into $3$ components, and at this point we can not yet exclude the 
possibility that $\Gamma$ consists of two disjoint cycles). For $1\leq i<j\leq 3$, let $\Gamma_{ij}$ be the 
subgraph of $\Gamma$ made of vertices and edges incident to $f_i$ and $f_j$. 
Let $a$ be the minimum label over all
vertices of $\Gamma$. 
By similar arguments as in the first paragraph of the proof of Proposition~\ref{prop:bij_bipointed},  
in case $d_{12}=s+t,d_{13}=s+u,d_{23}=t+u$ we must have $a\leq 0$, and in case $d_{12}=s+t-1,d_{13}=s+u-1,d_{23}=t+u-1$ we must either have 
$a<0$ or have $a=0$
and there is an edge in $\Gamma$ of labels $0-0$.  

Next we prove an upper bound for the two-point distance between two of the marked vertices in terms
of $a$. Two cases can arise: (i) no border edge has labels $a-a$, (ii) at least one border-edge
has labels $a-a$. In case (i) let $v$ be a border-vertex of label $a$, say without loss of generality
that $v$ is incident to $f_1$ and $f_2$. By the same arguments (using Lemma~\ref{lem:neighbour}) as in the proof of Proposition~\ref{prop:bij_bipointed}, there is a label-decreasing path in $f_1$ from $v$ to $v_1$,
and  there is a label-decreasing path in $f_2$ from $v$ to $v_2$, so that the concatenation of these
two paths yields a path of length $s+t+2a$ between $v_1$ and $v_2$. In case (ii) let $e'$ be a border-edge of labels $a-a$, and let $e$ be the edge of labels $(a-1)-(a-1)$ in $M$ that is dual to $e'$. 
Say without loss of generality
that $e'$ is incident to $f_1$ and $f_2$. Again by the same arguments (using Lemma~\ref{lem:neighbour}) as in the proof of Proposition~\ref{prop:bij_bipointed}, 
there is a path of length $s+t+2a-1$ (passing by $e$) 
between $v_1$ and $v_2$. Hence, together with the first paragraph, we conclude that, if the distances satisfy 
$d_{12}=s+t,d_{13}=s+u,d_{23}=t+u$, then  
 we must have $a=0$ and there is no border-edge of labels $0-0$,
and if the distances satisfy  $d_{12}=s+t-1,d_{13}=s+u-1,d_{23}=t+u-1$, then we must have 
$a=0$ and there is a border-edge of labels $0-0$. 

Assume $d_{12}=s+t,d_{13}=s+u,d_{23}=t+u$. We now prove that, for any $1\leq i<j\leq 3$, $\Gamma_{ij}$
contains a vertex of label $0$. Take (without loss of generality) $i=1$ and $j=2$, and consider
a path $P$ of length $s+t$ between $v_1$ and $v_2$. Assume for contradiction that
 $P$ does not meet $\Gamma_{12}$.  
Let $p$ be the first point of intersection of $P$ with $\Gamma$ and let $p'$ be the last point
of intersection of $P$ with $\Gamma$. Since $P$ does not meet $\Gamma_{12}$, $p$ is in
$\Gamma_{13}$ and not incident to $f_2$, and $p'$ is in $\Gamma_{23}$ and not incident to $f_1$;
in particular $p\neq p'$. Note that $p$ is either 
 at a vertex of label $b$ with $b\geq 0$ or at the middle of  an edge of $\Gamma_{13}$ 
of labels $b-b$ with $b\geq 1$; and similarly $p'$ is either 
 at a vertex of label $c$ with $c\geq 0$ or at the middle of  an edge of $\Gamma_{23}$ 
of labels $c-c$ with $c\geq 1$. 
In all four cases it is easily checked that the length of $P$ is strictly larger than $s+t$, 
which yields a contradiction. Hence $P$ has to cross $\Gamma_{12}$, either at a vertex 
of label $b\geq 0$, or at an edge of labels $b-b$ with $b\geq 1$. In the second case, $P$ has length at least $s+t+2b-1>s+t$, which is impossible, and in the first case $P$ has length at least $s+t+2b$, so that $b=0$. We conclude
that $\Gamma_{12}$ (and more generally any $\Gamma_{ij}$) has a vertex of label $0$. (Note that, in particular, all $\Gamma_{ij}$ are non empty, which excludes the possibility that $\Gamma$ is made
of two disjoint cycles, hence $\Gamma$ is connected.)  We have thus shown that, if  $d_{12}=s+t,d_{13}=s+u,d_{23}=t+u$, then $M'$ is of type A. 
Similarly, if $d_{12}=s+t-1,d_{13}=s+u-1,d_{23}=t+u-1$, one can  
prove that, for any $1\leq i<j\leq 3$, $\Gamma_{ij}$
contains an edge of labels $0-0$ (again by considering a geodesic path from $v_i$ to $v_j$,
then showing that $P$ must have length larger than $s+t-1$ in all cases, except for 
the case where $P$ crosses $\Gamma_{ij}$ at an edge of labels $0-0$).  Hence $M'$ is of type B.

Conversely, if $M'$ is of type A, then one can check at first that the distances satisfy $d_{12}\geq s+t,d_{13}\geq s+u,d_{23}\geq t+u$
(by the arguments in the first paragraph of the proof of Proposition \ref{prop:bij_bipointed}). Then one can also check that $d_{12}\leq s+t,d_{13}\leq s+u,d_{23}\leq t+u$ (for $d_{12}$, take a vertex of label $0$ on $\Gamma_{12}$, and construct a label-decreasing
path from $v$ to $v_1$, of length $s$, and a label-decreasing
path from $v$ to $v_2$, of length $t$). Hence if $M'$ is of type A, then the distances satisfy  $d_{12}=s+t,d_{13}=s+u,d_{23}=t+u$. Similarly, if $M'$ is of type B, then the distances satisfy  $d_{12}=s+t-1,d_{13}=s+u-1,d_{23}=t+u-1$.
\end{proof}

\begin{figure}
\begin{center}
\includegraphics[width=13cm]{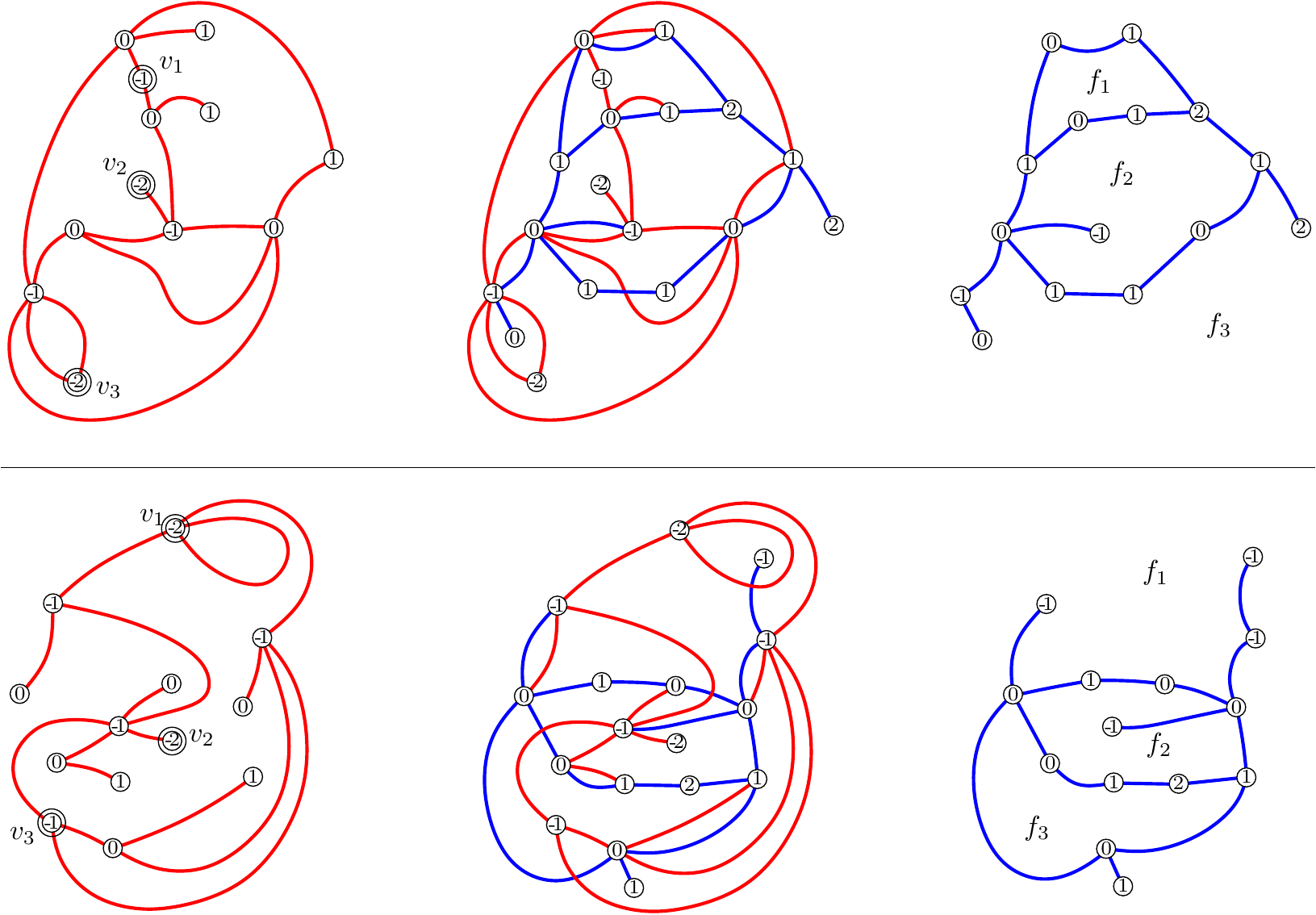}
\end{center}
\caption{Top: the left-part shows a tri-pointed map of distances $d_{12}=3,d_{13}=3,d_{23}=4$, endowed with its $[1,2,2]$-well-labelling; the right-part shows the associated $(1,2,2)$-well-labelled map of type A (and the middle-part shows both drawings superimposed). Bottom: the left-part shows a tri-pointed map of distances $d_{12}=3,d_{13}=2,d_{23}=2$, endowed with its $[2,2,1]$-well-labelling; the right-part shows the associated $(2,2,1)$-well-labelled map of type B 
(and the middle-part shows both drawings  superimposed).}
\label{fig:bij_tripointed}
\end{figure}

The proposition then yields the following bijection (see Figure~\ref{fig:bij_tripointed} for examples):

\begin{theo}\label{theo:bij_tpointed}
Let $d_{12},d_{13},d_{23}$ be a triple of positive integers satisfying triangular inequalities. 

(A): If  $d_{12}+d_{13}+d_{23}$ is even and $d_{12},d_{13},d_{23}$ satisfy strict triangular inequalities, let $s,t,u$ be the unique triple of positive integers~\footnote{given by $s=(d_{12}+d_{13}-d_{23})/2, t=(d_{12}+d_{23}-d_{13})/2, u=(d_{13}+d_{23}-d_{12})/2$.}
such that
$$
d_{12}=s+t,\ \ d_{13}=s+u,\ \ d_{23}=t+u.
$$
Then tri-pointed maps with distances $d_{12},d_{13},d_{23}$ are in bijection
(via $[s,t,u]$-well-labelled maps) with $(s,t,u)$-well-labelled maps of type A. The tri-pointed map
is bipartite iff the corresponding $(s,t,u)$-well-labelled map of type A is very-well-labelled.

(B): If  $d_{12}+d_{13}+d_{23}$ is odd, let $s,t,u$ be the unique triple of positive integers~\footnote{given by $s=(d_{12}+d_{13}-d_{23}+1)/2, t=(d_{12}+d_{23}-d_{13}+1)/2, u=(d_{13}+d_{23}-d_{12}+1)/2$.}
such that:
$$
d_{12}=s+t-1,\ \ d_{13}=s+u-1,\ \ d_{23}=t+u-1.
$$
Then tri-pointed maps with distances $d_{12},d_{13},d_{23}$ are in bijection
(via $[s,t,u]$-well-labelled maps) with $(s,t,u)$-well-labelled maps of type B.

In both bijections (A) and (B), each face of the tri-pointed map corresponds to a local max in 
the associated $(s,t,u)$-well-labelled map.
\end{theo}
\begin{proof}
It suffices to show that a tri-pointed map $M$ with distances $s+t,s+u,t+u$ (resp.\ $s+t-1,s+u-1,t+u-1$) 
can be
uniquely labelled as an $[s,t,u]$-well-labelled map, and then to apply 
Proposition~\ref{prop:bij_tripointed}. 
Consider the labelling-assignment $\ell_0(v)=\mathrm{min} (d(v,v_1)-s,d(v,v_2)-t,d(v,v_3)-u)$, where $d(v,w)$ denotes the distance (in $M$) between $v$ and $w$.  Clearly $\ell_0$ is a well-labelling. 
Similarly as in the proof of Theorem~\ref{theo:bij_bipointed} it can be checked that $\ell_0$ 
is indeed the unique $[s,t,u]$-well-labelling of~$M$. 

Moreover, for a tri-pointed map $M$ with distances $s+t,s+u,t+u$, $\ell_0$ is a very-well labelling
iff $M$ is bipartite (indeed, if $M$ is bipartite, for $v$ a vertex of $M$, either 
$d(v,v_1)-s,d(v,v_2)-t,d(v,v_3)-u$ are all even, or $d(v,v_1)-s,d(v,v_2)-t,d(v,v_3)-u$ are all odd, depending on
whether $v$ is in one vertex-color or the other).  
\end{proof}

\begin{rmk}
The case where the triangular inequalities on $d_{12},d_{13},d_{23}$ are not strict is not covered
by Theorem~\ref{theo:bij_tpointed} but by Remark~\ref{rk:bi_pointed} after Theorem~\ref{theo:bij_bipointed}. 
Indeed, let again $s,t,u$ be the values
such that $d_{12}=s+t,d_{13}=s+u,d_{23}=t+u$. 
When 
the triangular inequalities are not strict, exactly one of $s,t,u$ is zero (note that at most one of $s,t,u$ can be zero for the three marked vertices to be distinct), say $u=0$.  In that case, $v_3$ is a vertex
on a geodesic path between $v_1$ and $v_2$, at distance $s$ from $v_1$. By Remark~\ref{rk:bi_pointed}, 
tri-pointed maps
with distances $d_{12}=s+t, d_{13}=s, d_{23}=t$ are in bijection with $(s,t)$-well-labelled maps of type A 
with a marked border-vertex $v_3$ of label~$0$.
\end{rmk}

\section{Two- and three-point functions}
\label{sec:genfunc}

We shall now make use of Theorems \ref{theo:bij_bipointed} and \ref{theo:bij_tpointed} to derive
explicit expressions for the two- and three-point functions of general and bipartite planar maps. Recall that, by two-point 
(resp.\ three-point) function, we mean, in all generality, the generating function of bi-pointed (resp.\ tri-pointed) maps 
with a prescribed two-point distance $d_{12}$ (resp.\ with prescribed pairwise distances $d_{12}$, $d_{13}$ and $d_{23}$). 
We shall first concentrate on the two- and three-point functions of {\it general planar maps}, enumerated with a weight $g$ {\it per edge} ,
and then consider the subclass of {\it bipartite planar maps}. Note that bi-pointed maps may have a $k$-fold symmetry by ``rotation" around their 
two marked vertices: as customary, these $k$-fold symmetric maps receive an extra conventional 
weight $1/k$ (only with this convention is the two-point function simple). In the following,
we shall denote by $G_{d_{12}}\equiv G_{d_{12}}(g)$ and $G_{d_{12},d_{13},d_{23}}
\equiv G_{d_{12},d_{13},d_{23}}(g)$ (resp.\ by  $\tilde{G}_{d_{12}}$ and $\tilde{G}_{d_{12},d_{13},d_{23}}$)
the two- and three-point functions of general (resp.\ bipartite) planar maps. 

\subsection{Two-point function of general maps}
\label{sec:gentwopoint}
The two- and three-point functions may be expressed in terms of a number of generating functions for suitably defined
well-labelled objects which we will introduce along this paper as we need them.
A first basic ingredient is the generating function $T_s\equiv T_s(g)$ ($s>0$) of $(s)^+$-well-labelled maps, enumerated with a weight 
$g$ per edge, with a {\it marked corner} at a vertex labelled $0$. 
Here, by $(s)^+$-well-labelled map, we mean a well-labelled map with a single face $f$ (i.e., a tree) such that $\min(f)\geq-s+1$.

Introducing the notation
\begin{equation}
[s]_x \equiv 1-x^s\ ,
\label{eq:crocheti}
\end{equation}
the following expression for $T_s$ has been known for quite a while, obtained by various techniques \cite{GEOD,BG12}
\begin{equation}
\begin{split}
&T_s=T \frac{[s]_x[s+3]_x}{[s+1]_x[s+2]_x} \\
&\text{where}\ \ T=\frac{1+4x+x^2}{1+x+x^2}\ \ \text{and}\ \ g=x \frac{1+x+x^2}{(1+4x+x^2)^2}\ .
\label{eq:2pgenTi}
\end{split}
\end{equation}
In the above parametrization of $g$ (which is symmetric under $x\to 1/x$) and in similar expressions below, 
we always pick for $x$ the solution having modulus less than $1$ near $g=0$ (in particular $x=O(g)$).  
Note then that, for $s>0$, $T_s=1+O(g)$ with a first term $1$ accounting for the ``vertex-map" reduced to a single vertex with label $0$.
Note also that the above expression for $T_s$ may formally be extended to the case $s=0$ as it yields
$T_0=0$, which is the wanted result. 
 
An explicit expression for the two-point function of general planar maps was already obtained in \cite{AmBudd,BFG}.
It follows from a bijective coding of maps with two marked vertices at distance $d_{12}$ by $(s)$-well-labelled maps 
(well-labelled maps with a single face $f$ satisfying $\min(f)=-s+1$) with $s=d_{12}$, with a marked vertex 
labelled $0$ which is not a local max. It was found that,
for $d_{12}\geq 1$
 \begin{equation}
 G_{d_{12}}= \log\left(\frac{1+g\, T_{s}T_{s+1}}{1+g \, T_{s-1}T_{s}}\right)
 =\log\left(\frac{([s+1]_x)^3[s+3]_x}{[s]_x([s+2]_x)^3}\right)\quad \text{with}\ s=d_{12}\ .
 \label{eq:twopoint}
 \end{equation}
 
\begin{figure}[h!]
\begin{center}
\includegraphics[width=7cm]{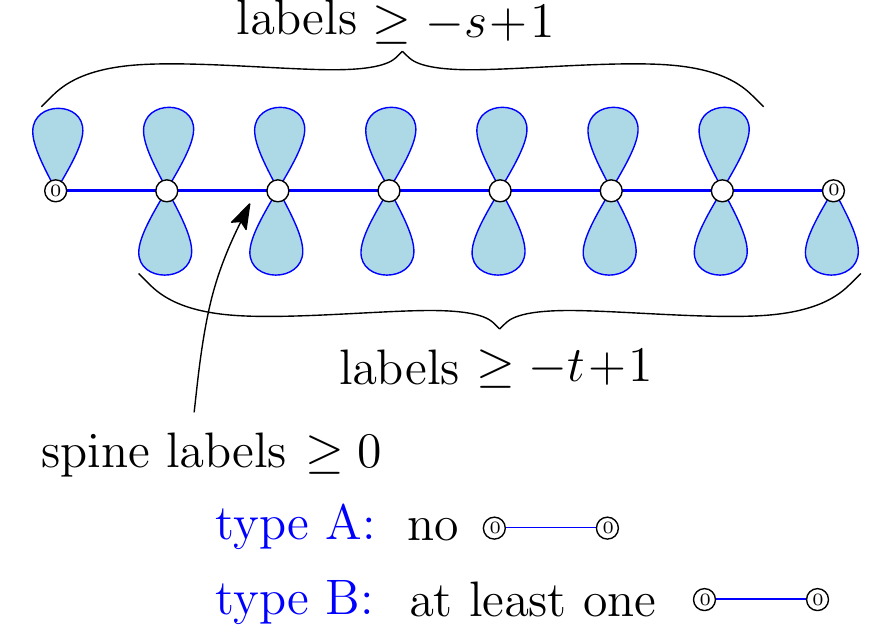}
\end{center}
\caption{Schematic picture of an $(s,t)^+$-well-labelled chain. The blobs represent attached well-labelled subtrees.
The labels of the spine-vertices are $0$ at the extremities of the spine and non-negative in-between.
The chain is of type A if it has no spine-edge with labels $0-0$, and of type B otherwise.} 
\label{fig:chains}
\end{figure}

Our new bijections (A) and (B) of Theorem \ref{theo:bij_bipointed} provide two alternative routes to recover this two-point 
function. This new approach is instructive as it will lead us to introduce and evaluate new generating functions
which will turn out to be useful to later compute the three-point function. 

The first generating function of interest is that, $N_{s,t}\equiv N_{s,t}(g)$, 
of $(s,t)^+$well-labelled {\it chains} of type A, enumerated with a weight $g$ per edge, defined as follows: by well-labelled chain, 
we mean a well-labelled tree made of a distinguished linear (oriented) spine of some arbitrary length, with additional subtrees attached to
the spine-vertices (see Figure~\ref{fig:chains} for an illustration). 
We have an $(s,t)^+$-well-labelled chain if (i) the labels of the spine-vertices are $0$ at the extremities of the spine and non-negative in-between and
(ii) the vertices of the subtrees attached to the first extremity of the spine and to the left of 
all inner spine-vertices have a label $\geq-s+1$ and the vertices of the subtrees attached to the last extremity of the spine 
and to the right of all inner spine-vertices have a label 
$\geq -t+1$. Finally an $(s,t)^+$-well-labelled chain is said to be of type $A$ if there is no spine-edge with labels $0-0$. 
If the spine has length $0$, i.e., reduces to a single vertex, we decide by convention 
{\it not to attach any subtree} at all to this vertex: this configuration receives a weight $1$ accordingly. 
The above definition of $N_{s,t}$ holds for $s,t>0$ and we extend it for convenience by setting $N_{s,0}=N_{0,t}=N_{0,0}=1$. 

Similarly, we may introduce the generating function $O_{s,t}\equiv O_{s,t}(g)$ of $(s,t)^+$-well-labelled chains of type B
defined as $(s,t)^+$-well-labelled chains with at least one spine-edge with labels $0-0$
(again this definition holds for $s,t>0$ and we extend it by now setting $O_{s,0}=O_{0,t}=O_{0,0}=0$).

From the bijection (A) of Theorem \ref{theo:bij_bipointed}, when $d_{12}\geq 2$, $G_{d_{12}}$ is also the generating function of 
$(s,t)$-well-labelled maps of type A, enumerated with a weight $g$ per edge, for any pair of positive integers $s$ and $t$ satisfying $s+t=d_{12}$.
At this stage, it is useful to weaken the definition of $(s,t)$-well-labelled maps of type A and consider instead what we shall call
$(s,t)^+$-well-labelled maps of type A, satisfying now the weaker conditions $\min(f_1)\geq -s+1$ and $\min(f_2)\geq -t+1$
for the labels incident to their faces $f_1$ and $f_2$. In other words, an $(s,t)^+$-well-labelled map is an $(s',t')$-well-labelled
map with $0<s'\leq s$ and $0<t'\leq t$. Focusing on the cycle $\Gamma$ formed by their border-vertices and border-edges,
we then remark that $(s,t)^+$-well-labelled maps of type A may be seen as cyclic sequences of elementary blocks (which consist of
$(s,t)^+$-well-labelled 
chains of type A having no spine-label $0$ between the extremities of its spine) while $(s,t)^+$-well-labelled chains of type A 
correspond to (linear) sequences of blocks of the same type. By a standard argument, we immediately deduce that
the generating function for $(s,t)^+$-well-labelled maps of type A is given by $\log(N_{s,t})$ and consequently, that that of
$(s,t)$-well labelled maps of type A is given by $\Delta_s\Delta_t \log(N_{s,t})$. Here $\Delta_s$ is the finite difference operator $\Delta_s f(s)\equiv f(s)-f(s-1)$. Applying $\Delta_s$ on the generating function $\log(N_{s,t})$ indeed selects those configurations satisfying 
$-s+2>\min(f_1)\geq -s+1$, hence $\min(f_1)=-s+1$ and similarly, applying $\Delta_t$ imposes $\min(f_2)=-t+1$.
The finite difference operators therefore ensure the passage from $(s,t)^+$-well-labelled objects to $(s,t)$-well-labelled ones. We 
shall use this standard trick in various occasions later in the paper.
We deduce the relation, for $s,t>0$
\begin{equation}
G_{d_{12}}=\Delta_s\Delta_t \log(N_{s,t})=\log\left(\frac{N_{s,t}N_{s-1,t-1}}{N_{s-1,t}N_{s,t-1}}\right)\quad \text{with}\ s+t=d_{12}
\label{eq:twopointbis}
\end{equation}
(note that this holds for $s=1$ and $t=1$ thanks to our convention $N_{s,0}=N_{0,t}=N_{0,0}=1$).

To evaluate $N_{s,t}$, we may rely on a known formula for the generating function $X_{s,t}=N_{s,t}+O_{s,t}$ of $(s,t)^+$-well-labelled chains
of type A or B (i.e., with or without spine-edges of labels $0-0$). This later generating function 
was indeed computed in \cite{BG08}, with the result
\begin{equation}
X_{s,t}=\frac{[3]_x[s+1]_x[t+1]_x[s+t+3]_x}{[1]_x[s+3]_x[t+3]_x[s+t+1]_x}
\end{equation}
and satisfies the following recursion relation (obtained by decomposing the chain at its first return at label $0$ along the spine)
\begin{equation}
X_{s,t}=1+g\, T_s T_t X_{s,t}+g^2\, T_sT_t X_{s,t}T_{s+1}T_{t+1}X_{s+1,t+1}\ .
\label{eq:recurX}
\end{equation}
Now, by a decomposition of the chain at the spine-edges with labels $0-0$, we may write the relation $X_{s,t}=N_{s,t}/(1-g\, T_{s}T_{t}N_{s,t})$,
or equivalently
\begin{equation}
N_{s,t}=\frac{X_{s,t}}{1+g\, T_{s}T_{t}X_{s,t}}\ .
\label{eq:XtoN}
\end{equation}
Replacing $X_{s,t}$ by its value above, this leads to the particularly simple (and remarkably similar) expression
\begin{equation}
N_{s,t}=\frac{[3]_x[s+2]_x[t+2]_x[s+t+3]_x}{[2]_x[s+3]_x[t+3]_x[s+t+2]_x}\ .
\label{eq:Nst}
\end{equation}
Plugging this latter expression in \eqref{eq:twopointbis}, we get explicitely 
\begin{equation}
 G_{d_{12}}=
 \log\left(\frac{([s+t+1]_x)^3[s+t+3]_x}{[s+t]_x([s+t+2]_x)^3}\right)\quad \text{with}\ s+t=d_{12}\ ,
\end{equation}
which reproduces precisely the previous formula \eqref{eq:twopoint}, as wanted.

The knowledge of the generating function $N_{s,t}$ will be crucial in the derivation 
of the three-point function in the next section. It is interesting to note that its above expression \eqref{eq:Nst}
may be obtained in several alternative ways, without recourse to the known expression for $X_{s,t}$.
These alternative approaches will prove useful when we shall discuss similar generating functions for which we cannot
rely on known formulas.

First, we note that equating {\it ab initio} \eqref{eq:twopoint} and \eqref{eq:twopointbis} provides in return
a constructive way of getting $N_{s,t}$. Indeed, it allows us to write 
\begin{equation}
\frac{N_{s,t}N_{s-1,t-1}}{N_{s-1,t}N_{s,t-1}}=\frac{R_{s+t}}{R_{s+t-1}}\ , \qquad R_{u}\equiv 1\!+g\, T_{u}T_{u+1} = 
\frac{([2]_x)^2}{[1]_x[3]_x} \frac{[u+1]_x[u+3]_x}{([u+2]_x)^2}\ ,
\end{equation}
which is a double (in $s$ and $t$) recursion formula. Together with the conditions $N_{s,0}=N_{0,t}=N_{0,0}=1$, it leads to
\begin{equation}
N_{s,t}=\prod\limits_{u=1}^{s+t}R_u \Big/ \left(\prod\limits_{u=1}^{s}R_u\prod\limits_{u=1}^{t}R_u\right)\ ,
\end{equation}
which yields immediately \eqref{eq:Nst} by replacing $R_u$ by its value.

\begin{figure}[h!]
\begin{center}
\includegraphics[width=10cm]{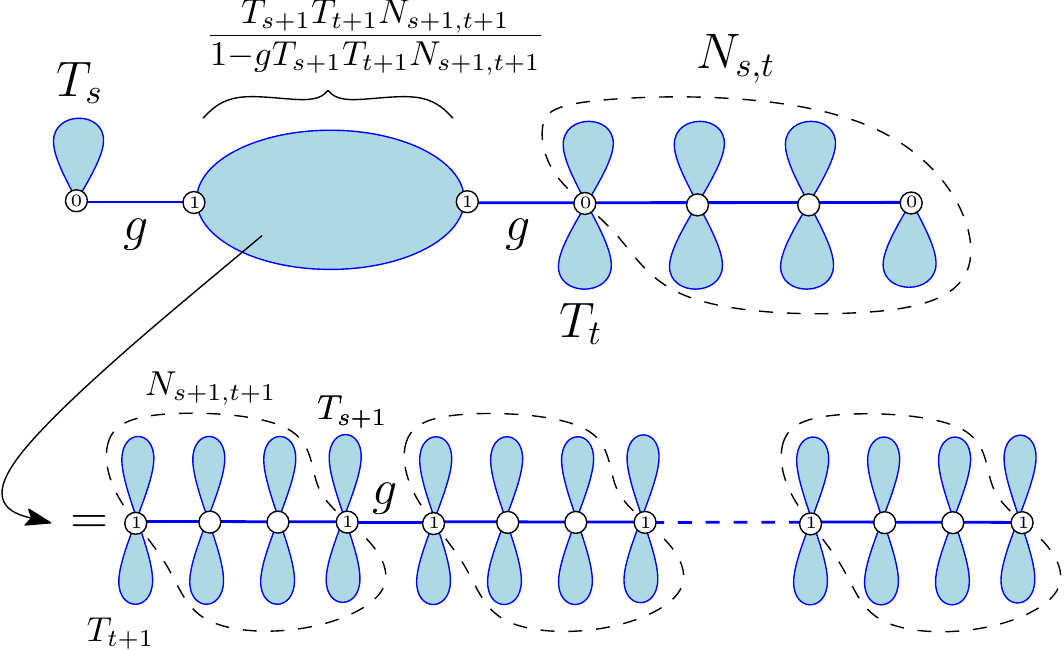}
\end{center}
\caption{A schematic picture of the recursion relation for $N_{s,t}$, obtained by decomposing an $(s,t)^+$-well-labelled chain of type A
at its first return at label $0$ along the spine.} 
\label{fig:nstdecomp}
\end{figure}

A second alternative, but non-constructive way of getting $N_{s,t}$ is by solving yet another recursion relation, which
is the analog of \eqref{eq:recurX},  obtained 
by a simple decomposition of an $(s,t)^+$-well-labelled chain of type A at its first return at label $0$ along the spine. 
This decomposition leads to (see Figure~\ref{fig:nstdecomp})
\begin{equation}
N_{s,t}=1+g^2 T_s T_t N_{s,t} \frac{T_{s+1}T_{t+1}N_{s+1,t+1}}{1-g\, T_{s+1}T_{t+1}N_{s+1,t+1}}\ .
\label{eq:recurN}
\end{equation}
Indeed, when not reduced to a single vertex (weight $1$), an $(s,t)^+$-well-labelled chain of type A has
a spine made of a first $0-1$ edge (weight $g\,T_s$, including the attached subtree), a portion of spine with labels larger than or equal to $1$, then
a first $1-0$ edge (weight $g\, T_t$) and a final portion which is itself an $(s,t)^+$-well-labelled chain of type A
(weight $N_{s,t}$). The portion of spine with labels larger than or equal to $1$ is, after a simple shift of labels by $-1$,
an $(s+1,t+1)^+$-well-labelled chain (of arbitrary type A or B) with two extra attached subtrees, as enumerated by $T_{s+1}T_{t+1}X_{s+1}
=T_{s+1}T_{t+1}N_{s+1,t+1}/(1-g\, T_{s+1}T_{t+1}N_{s+1,t+1})$. Note that eq.~\eqref{eq:recurN} may equivalently be obtained 
by simply plugging $X_{s,t}=N_{s,t}/(1-g\, T_{s}T_{t}N_{s,t})$ into \eqref{eq:recurX}.
Knowing the expression of $T_s$, this equation determines entirely $N_{s,t}$ as a power series in $g$. It is now a straightforward 
exercise to check that the above expression \eqref{eq:Nst} does indeed solve this equation (and is such that $N_{s,t}=1+O(g)$)
hence provides the correct expression for $N_{s,t}$. In the following, we shall recourse in several occasions to this method, i.e., write
down a recursion relation and {\it guess} its (unique as a power series in $g$) solution.

We may now easily repeat the above arguments by using instead the bijection (B) of Theorem \ref{theo:bij_bipointed}. For $d_{12}\geq 1$,
 $G_{d_{12}}$ is then identified with the generating function of $(s,t)$-well-labelled maps of type B,
enumerated with a weight $g$ per edge, for any pair of positive integers $s$ and $t$ such that $d_{12}=s+t-1$. 
Now, focusing on the cycle $\Gamma$ formed by their border-vertices and border-edges, $(s,t)^+$-well-labelled maps of type B 
are nothing but cyclic sequences where each elementary block
is made of an edge with labels $0-0$ (weight $g$) followed by an $(s,t)^+$-well-labelled chain of type A with two extra attached subtrees
(weight $T_{s}T_{t}N_{s,t}$). This allows us to write immediately
\begin{equation}
\begin{split}
G_{d_{12}}& =\Delta_s\Delta_t \log\left(\frac{1}{1-g\, T_{s}T_{t}N_{s,t}}\right)=
\Delta_s\Delta_t \log\left(\frac{X_{s,t}}{N_{s,t}}\right) \\
&= \log\left(\frac{X_{s,t}X_{s-1,t-1}N_{s-1,t}N_{s,t-1}}{X_{s-1,t}X_{s,t-1}N_{s,t}N_{s-1,t-1}}\right)\ . \\
\end{split}
\label{eq:twopointter}
\end{equation}
Replacing $X_{s,t}$ and $N_{s,t}$ by their expressions, we deduce
\begin{equation}
 G_{d_{12}}=
 \log\left(\frac{([s+t]_x)^3[s+t+2]_x}{[s+t-1]_x([s+t+1]_x)^3}\right)\quad \text{with}\ s+t-1=d_{12}
\end{equation}
in agreement with \eqref{eq:twopoint}, as wanted.

To end this section, let us finally give an expression for the generating function $O_{s,t}$
of $(s,t)^+$-well-labelled chains of type B. It is obtained via $O_{s,t}=X_{s,t}-N_{s,t}$, from which we deduce
\begin{equation}
O_{s,t}=x\, \frac{[3]_x[s]_x[t]_x([s+t+3]_x)^2}{[2]_x[s+3]_x[t+3]_x[s+t+1]_x [s+t+2]_x}\ .
\label{eq:Ost}
\end{equation}
\subsection{Three-point function of general maps}
\label{sec:genthreepoint}

\begin{figure}[h!]
\begin{center}
\includegraphics[width=10cm]{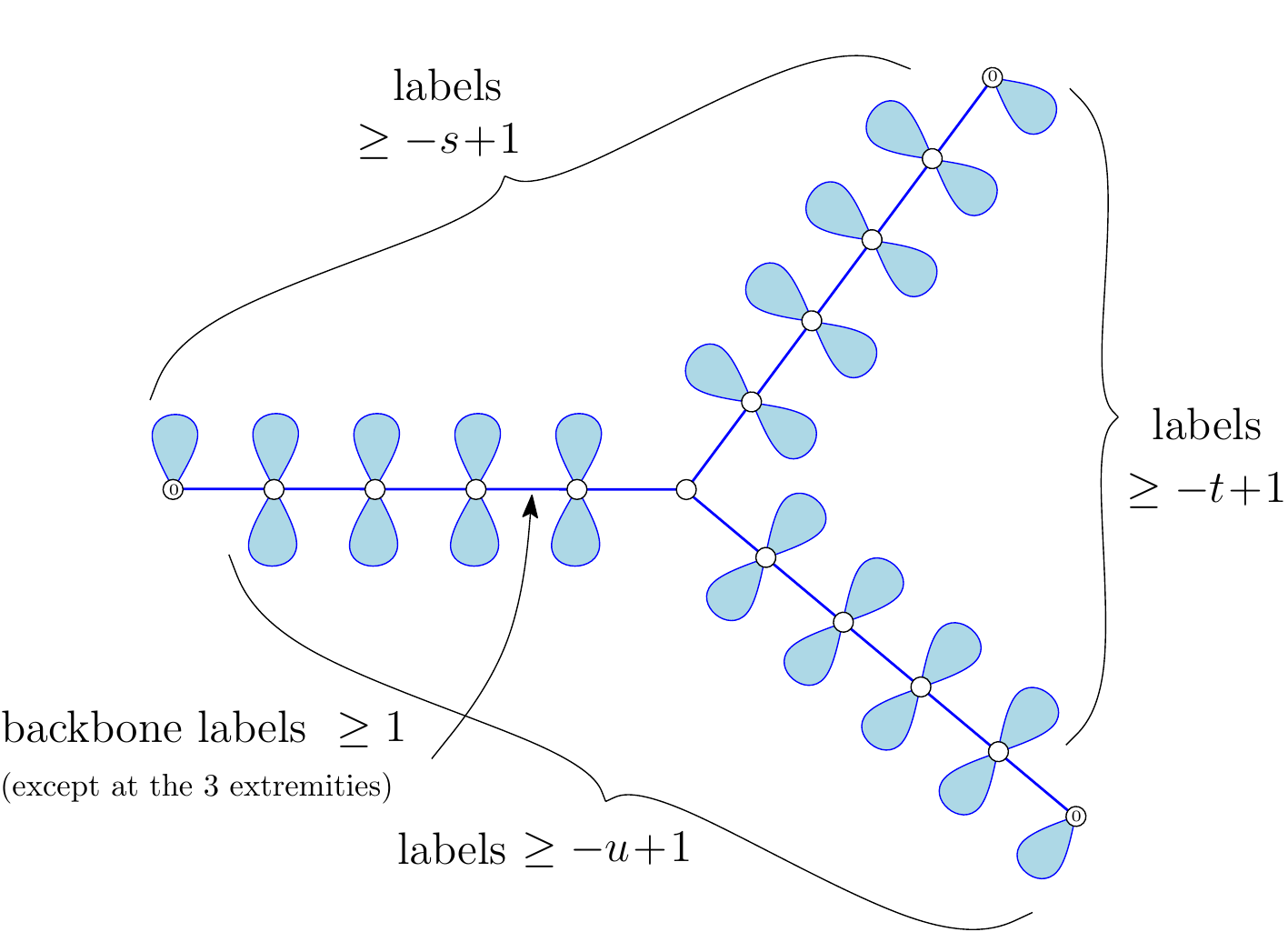}
\end{center}
\caption{Schematic picture of an $(s,t,u)^+$-well-labelled Y-diagram. The blobs represent attached well-labelled subtrees.
The label of the backbone-vertices are $0$ at the extremities of the three branches and positive in-between.} 
\label{fig:Ydiagrams}
\end{figure}

We shall now use the bijections of Theorem \ref{theo:bij_tpointed} to derive an explicit formula for the 
three-point function of general planar maps.
As a new tool, we shall need an expression for the generating function $Y_{s,t,u}\equiv Y_{s,t,u}(g)$ of
$(s,t,u)^+$-well-labelled {\it Y-diagrams} defined as follows, for $s,t,u>0$: by well-labelled Y-diagram, 
we mean a well-labelled tree with a distinguished backbone made of three branches (referred to as the first, second and third branch clockwise) 
of arbitrary lengths connected at a central vertex, with additional subtrees attached to the backbone-vertices 
(see Figure~\ref{fig:Ydiagrams} for an illustration). 
The labels of the backbone vertices are required to be $0$ at the extremities of the three branches and strictly positive in-between. 
We have an $(s,t,u)^+$-Y-diagram if, going clockwise around the backbone, the vertices of the subtrees attached to backbone-corners
lying between the extremity of the first branch (extremity included) and that of the second branch (extremity excluded) have a label 
$\geq -s+1$, the vertices of the subtrees attached to backbone-corners lying between the extremity of the second branch (extremity included) and that of the third branch (extremity excluded) have a label 
$\geq -t+1$ and, finally, the vertices of the subtrees attached to backbone-corners lying between the extremity of the third branch (extremity included) and that of the first branch (extremity excluded) have a label $\geq -u+1$.
Note that an $(s,t,u)^+$- well-labelled Y-diagram cannot contain edges with labels $0-0$ along its three branches since 
only the extremities of the branches carry a label $0$. Note also that the lengths of the three branches may be zero {\it simultaneously}
in with case the backbone of the Y-diagram reduces to a single vertex with label $0$: as before, we then decide for convenience  
not to attach any subtree at all to this vertex and this configuration receives a weight $1$ accordingly. 

The generating function $Y_{s,t,u}$ was already introduced in \cite{BG08}. There it was shown that 
it satisfies the following recursion relation, easily obtained by decomposing each branch of the Y-diagram at its first passage
(starting from the central vertex) at label $1$ along the branch
\begin{equation}
Y_{s,t,u}=1+g^3\, T_s T_t T_u X_{s+1,t+1} X_{s+1,u+1} X_{t+1,u+1}T_{s+1}T_{t+1}T_{u+1}Y_{s+1,t+1,u+1}\ .
\label{eq:recurY}
\end{equation}
Knowing $T_s$ and $X_{s,t}$, this equation determines $Y_{s,t,u}$ entirely as a power series in $g$. 
The following explicit solution was then found in \cite{BG08}
\begin{equation}
Y_{s,t,u}=\frac{[s+3]_x[t+3]_x[u+3]_x[s+t+u+3]_x}{[3]_x[s+t+3]_x[t+u+3]_x[u+s+3]_x}\ .
\label{eq:Ystu}
\end{equation}

\begin{figure}[h!]
\begin{center}
\includegraphics[width=8cm]{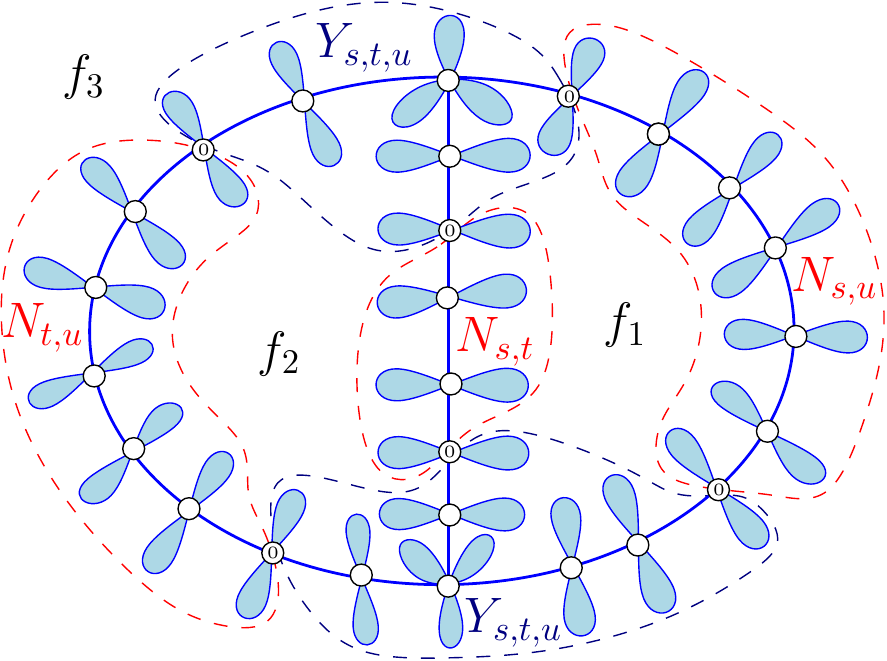}
\end{center}
\caption{Schematic picture of an $(s,t,u)^+$-well-labelled map. The blobs represent attached well-labelled subtrees.
The map is decomposed into five pieces, two of them enumerated by $Y_{s,t,u}$ and the last three by $N_{s,t}$, $N_{s,u}$ and $N_{t,u}$
respectively (see text).} 
\label{fig:stumapdecomp}
\end{figure}

We may now evaluate the three-point function $G_{d_{12},d_{13},d_{23}}$ of general maps.
Let us start with the case where $d_{12}+d_{13}+d_{23}$ is even and $d_{12}$, $d_{13}$ and $d_{23}$ satisfy strict 
triangular inequalities (we will return later to the case where triangular inequalities are not strict). As in the statement of
Theorem \ref{theo:bij_tpointed}-(A), we then set
\begin{equation}
d_{12}=s+t\ , \quad d_{13}=s+u\ , \quad d_{23}=t+u\ ,
\label{eq:dstupair}
\end{equation}
where $s$, $t$ and $u$ are three positive integers.
From the bijection (A) of Theorem \ref{theo:bij_tpointed}, $G_{d_{12},d_{13},d_{23}}$  is identified with the generating function
of $(s,t,u)$-well-labelled maps of type A, which allows us to write
\begin{equation}
G_{d_{12},d_{13},d_{23}}=\Delta_s\Delta_t\Delta_u N_{s,t}N_{s,u}N_{t,u}Y_{s,t,u}^2\ .
\label{eq:threepointfirst}
\end{equation}
Indeed, introducing $(s,t,u)^+$-well-labelled maps of type A (defined as  $(s,t,u)$-well-labelled maps of type A except
for the weaker constraints $\min(f_1)\geq -s+1$, $\min(f_2)\geq -t+1$ and $\min(f_3)\geq -u+1$), these maps may be 
decomposed into five pieces by cutting their ``backbone" at the first and last occurrence of a label $0$ along each  
border between their three faces.  More precisely, introducing as before the three subgraphs $\Gamma_{ij}$ of the map made of vertices and edges incident to
the faces $f_i$ and $f_j$ ($1\leq i<j\leq 3$), $\Gamma_{12}$, $\Gamma_{13}$ and $\Gamma_{23}$ form a backbone made generically of three
chains attached at their extremities to two ``triple-point" vertices (see Figure~\ref{fig:stumapdecomp}). Since 
the map is of type A, each $\Gamma_{ij}$ carries at least a label $0$ by definition. 
Cutting the chains at its first and last occurrence of such a label $0$ (as encountered by going from one triple-point to the other)
results in an $(s,t,u)^+$-well-labeled Y-diagrams and an $(s,u,t)^+$-well-labeled Y-diagrams (both enumerated 
by $Y_{s,t,u}$) and three chains: an $(s,t)^+$-well-labeled chain, an $(s,u)^+$-well-labelled chain
and a $(t,u)^+$-well-labelled chain, {\it all of type A}, hence enumerated by $N_{s,t}$, $N_{s,u}$ and $N_{t,u}$ respectively.
Finally, the passage from $(s,t,u)^+$-well-labelled maps to $(s,t,u)$-well-labelled maps is performed by the action 
of the finite difference operators $\Delta_s$, $\Delta_t$ and $\Delta_u$. As explained in \cite{BG08} (in a similar
calculation for quadrangulations),
degenerate situations where the border between two of the three faces is reduced to a single
vertex (so that the two triple-points coalesce) are properly enumerated by this formula. For instance, configurations whose $(1,2)$-border
is a single vertex are enumerated by $\Delta_s\Delta_t\Delta_u N_{s,u}N_{t,u}$. This contribution to $G_{d_{12},d_{13},d_{23}}$ properly appears in 
\eqref{eq:threepointfirst} by picking the first term $1$ in the expansion of both $N_{s,t}$ and $Y_{s,t,u}$.

Replacing $Y_{s,t,u}$ and $N_{s,t}$ by their explicit expressions, we obtain the following result:
\begin{prop}[Three-point function of general maps: even case]
Given $d_{12}$, $d_{13}$ and $d_{23}$ three positive integers satisfying strict triangular inequalities,
and such that $d_{12}+d_{13}+d_{23}$ is even, the three-point function $G_{d_{12},d_{13},d_{23}}$ 
is given by
\begin{equation}
\begin{split}
& \hskip -10pt G_{d_{12},d_{13},d_{23}} =\Delta_s\Delta_t\Delta_u  F^{\text{even}}_{s,t,u}\\
& \hskip -10pt F^{\text{even}}_{s,t,u}=\ \frac{[3]_x([s\!+\!2]_x[t\!+\!2]_x[u\!+\!2]_x[s\!+\!t\!+\!u\!+\!3]_x)^2}{([2]_x)^3[s\!+\!t+2]_x[t\!+\!u+2]_x[u\!+\!s+2]_x
[s\!+\!t\!+\!3]_x[t\!+\!u\!+\!3]_x[u\!+\!s\!+\!3]_x}\\
& \\
& \hskip 220pt \text{with}\ s,t,u\ \text{as in \eqref{eq:dstupair}}\ .\\
\label{eq:threepointgenpair}
\end{split}
\end{equation} 
\end{prop}
Let us now discuss the case where triangular inequalities are not strict by setting 
for instance $u=0$ in \eqref{eq:dstupair}, in which case $d_{13}+d_{23}=d_{12}$  and $v_3$ lies on a geodesic path between $v_1$ and $v_2$.
As explained in Remark 3, such tri-pointed maps are in bijection with $(s,t)$-well-labelled maps of type A with a marked 
border-vertex with label $0$. This marking transforms de facto (by a simple cut) the $(s,t)$-well-labelled map of type A 
into an $(s,t)$-well-labelled chain of type A, as enumerated by $\Delta_s\Delta_t N_{s,t}$. We deduce
\begin{equation} 
G_{d_{13}+d_{23},d_{13},d_{23}} =\Delta_s\Delta_t N_{s,t}\quad \text{with}\ s=d_{13}, \  t=d_{23}\ .
\label{eq:threealignedgen}
\end{equation}
Note that this latter expression is precisely that given by \eqref{eq:threepointgenpair} if we formally set 
$F^{\text{even}}_{s,t,-1}=0$ so that $\Delta_u F^{\text{even}}_{s,t,u}\vert_{u=0}= F^{\text{even}}_{s,t,0}=N_{s,t}$.  

Let us now come to the case where $d_{12}+d_{13}+d_{23}$ is odd and, as in the statement of Theorem \ref{theo:bij_tpointed}-(B), 
set
\begin{equation}
d_{12}=s+t-1\ , \quad d_{13}=s+u-1\ , \quad d_{23}=t+u-1
\label{eq:dstuimpair}
\end{equation}
where $s$, $t$ and $u$ are positive integers.
From the bijection (B) of Theorem \ref{theo:bij_tpointed}, $G_{d_{12},d_{13},d_{23}}$  is now identified with the generating function 
of $(s,t,u)$-well-labelled maps of type B, which allows us to write
\begin{equation}
G_{d_{12},d_{13},d_{23}}=\Delta_s\Delta_t\Delta_u O_{s,t}O_{s,u}O_{t,u}Y_{s,t,u}^2\ .
\end{equation}
Replacing $Y_{s,t,u}$ and $O_{s,t}$ by their explicit expressions, we deduce
\begin{prop}[Three-point function of general maps: odd case]
Given $d_{12}$, $d_{13}$ and $d_{23}$ three positive integers satisfying triangular inequalities,
and such that $d_{12}+d_{13}+d_{23}$ is odd, the three-point function $G_{d_{12},d_{13},d_{23}}$ 
is given by
\begin{equation}
\begin{split}
& \hskip -10pt G_{d_{12},d_{13},d_{23}} =\Delta_s\Delta_t\Delta_u  F^{\text{odd}}_{s,t,u}\\
& \hskip -10pt F^{\text{odd}}_{s,t,u}=x^3 \frac{[3]_x([s]_x[t]_x[u]_x[s\!+\!t\!+\!u\!+\!3]_x)^2}{([2]_x)^3[s\!+\!t\!+\!1]_x[t\!+\!u\!+\!1]_x[u\!+\!s\!+\!1]_x
[s\!+\!t\!+\!2]_x[t\!+\!u\!+\!2]_x[u\!+\!s\!+\!2]_x}\\
& \\
& \hskip 220pt \text{with}\ s,t,u\ \text{as in \eqref{eq:dstuimpair}}\ .\\
\label{eq:threepointgenimpair}
\end{split}
\end{equation} 
\end{prop}

As a simple application of our formulas, let us compute for instance the first terms in the small $g$ expansion
of $G_{2,2,2}$ (three vertices at pairwise distances $d_{12}=d_{13}=d_{23}=2$) and $G_{1,1,1}$ 
(three vertices at pairwise distances $d_{12}=d_{13}=d_{23}=1$). Both are obtained by setting
$s=t=u=1$, respectively in \eqref{eq:threepointgenpair} and \eqref{eq:threepointgenimpair}. 
From the relation between $g$ and $x$ in \eqref{eq:2pgenTi}, we deduce the expansion
\begin{equation}
x= g+7 g^2+59 g^3+544 g^4+5289 g^5+53256 g^6+549771 g^7+5782105 g^8+\ldots
\end{equation}
Plugging this expansion in \eqref{eq:threepointgenpair} and \eqref{eq:threepointgenimpair} yields
\begin{equation}
\begin{split}
G_{2,2,2} &= 2 g^3+39 g^4+558 g^5+7123 g^6+86139 g^7+1011954 g^8+\ldots\\
G_{1,1,1} &= g^3+15 g^4+174 g^5+1867 g^6+19482 g^7+201450 g^8+\ldots \\
\end{split}
\end{equation}
whose first terms (of order $g^3$ and $g^4$) are easily recovered by a simple inspection.

\subsection{Two- and three-point functions of bipartite maps}
\label{sec:biptwothreepoint}
We may easily repeat the arguments of Sects.~\ref{sec:gentwopoint} and \ref{sec:genthreepoint} to obtain the two- and three-point 
functions of bipartite planar maps. In practice, we simply have to consider the same generating functions as above restricted to the subclass
of {\it very-well-labelled} objects. Our first basic ingredient is therefore the generating function 
$\tilde{T}_s\equiv \tilde{T}_s(g)$ ($s>0$) of $(s)^+$-very-well-labelled maps, enumerated with a weight 
$g$ per edge, with a marked corner at a vertex labelled $0$. The following expression for $\tilde{T}_s$ was derived in \cite{GEOD}
\begin{equation}
\begin{split}
&\tilde{T}_s=\tilde{T} \frac{[s]_x[s+4]_x}{[s+1]_x[s+3]_x} \\
&\text{where}\ \ \tilde{T}=\frac{(1+x)^2}{1+x^2}\ \ \text{and}\ \ g=x \frac{1+x^2}{(1+x)^4}\ .
\label{eq:2pbipTi}
\end{split}
\end{equation}
As explained in \cite{BFG}, we may deduce from $\tilde{T}_s$ an explicit formula for the two-point function
of bipartite planar maps
\begin{equation}
 \tilde{G}_{d_{12}}= \log\left(\frac{1+g\, \tilde{T}_{s}\tilde{T}_{s+1}}{1+g \, \tilde{T}_{s-1}\tilde{T}_{s}}\right)
 =\log\left(\frac{([s+1]_x)^2[s+4]_x}{[s]_x([s+3]_x)^2}\right)\quad \text{with}\ s=d_{12}\ .
 \label{eq:twopointbip}
 \end{equation}
Let us now see how to recover this formula in the framework of bi-pointed maps.
From Theorem \ref{theo:bij_bipointed}, $\tilde{G}_{d_{12}}$ with $d_{12}=s+t$ is identified with the generating function of 
$(s,t)$-very-well-labelled maps
of type A, enumerated with a weight $g$ per edge. Note that being of type A simply amounts here to demanding 
that the minimum label over all border vertices in the map is $0$, and that there are no $(s,t)$-very-well-labelled maps
of type B. As for $(s,t)^+$-very-well-labelled chains, they are automatically of type A: if we insist in defining the very-well-labelled 
analogs $\tilde{N}_{s,t}$, $\tilde{O}_{s,t}$ and $\tilde{X}_{s,t}$ of $N_{s,t}$,
$O_{s,t}$ and $N_{s,t}$, we must then set  
\begin{equation}
\tilde{N}_{s,t}=\tilde{X}_{s,t}\quad \text{and}\quad \tilde{O}_{s,t}=0\ 
\end{equation}
so that, in practice, we only have to deal with a single generating function $\tilde{X}_{s,t}$, that of $(s,t)^+$-very-well labelled chains.
By the same argument as in Sect.~\ref{sec:gentwopoint},
we may now write the two-point function as
\begin{equation}
\tilde{G}_{d_{12}}=\Delta_s\Delta_t \log(\tilde{X}_{s,t})=\log\left(\frac{\tilde{X}_{s,t}\tilde{X}_{s-1,t-1}}{\tilde{X}_{s-1,t}\tilde{X}_{s,t-1}}\right)\quad \text{with}\ s+t=d_{12}\ .
\label{eq:twopointbipbis}
\end{equation}
Comparing with \eqref{eq:twopointbip}, this leads to the identification
\begin{equation}
\hskip -10.pt \tilde{X}_{s,t}=\prod\limits_{u=1}^{s+t}\tilde{R}_u \Big/ \left(\prod\limits_{u=1}^{s}\tilde{R}_u\prod\limits_{u=1}^{t}\tilde{R}_u\right)
\ , \quad \tilde{R}_{u}\equiv 1\!+g\, \tilde{T}_{u}\tilde{T}_{u+1} =\frac{[2]_x[3]_x}{[1]_x[4]_x}\frac{[u+1]_x[u+4]_x}{[u+2]_x[u+3]_x}
\end{equation}
hence to the explicit expression
\begin{equation}
\tilde{X}_{s,t}=\frac{[4]_x[s+2]_x[t+2]_x[s+t+4]_x}{[2]_x[s+4]_x[t+4]_x[s+t+2]_x}\ .
\label{eq:tildeXst}
\end{equation}
As a check of consistency, we now argue that $\tilde{X}_{s,t}$ is, alternatively, entirely determined as a power series in $g$
by the recursion 
\begin{equation}
\tilde{X}_{s,t}=1+g^2\, \tilde{T}_s \tilde{T}_t \tilde{X}_{s,t}\tilde{T}_{t+1}\tilde{T}_{s+1} \tilde{X}_{s_1,t+1}\ ,
\end{equation}
obtained by decomposing the chain at its first return to $0$ along the spine.
It is a simple exercise to check that \eqref{eq:tildeXst} actually solves this equation, as wanted.

Coming now to the three-point function $\tilde{G}_{d_{12},d_{13},d_{23}}$ of bipartite planar maps, 
we set again (recall that the sum $d_{12}+d_{13}+d_{23}$ is necessarily even in a bipartite map)
\begin{equation}
d_{12}=s+t\ , \quad d_{13}=s+u\ , \quad d_{23}=t+u\ 
\label{eq:dstubip}
\end{equation}
with $s$, $t$ and $u$ positive integers (we assume here strict triangular inequalities).
To get $\tilde{G}_{d_{12},d_{13},d_{23}}$, we now need an expression for the
generating function $\tilde{Y}_{s,t,u}\equiv \tilde{Y}_{s,t,u}(g)$ of 
$(s,t,u)^+$-very-well-labelled Y-diagrams (which form the very-well labelled subclass of $(s,t,u)^+$-well-labelled Y-diagrams).
Since no expression was known so far for $\tilde{Y}_{s,t,u}$, we had to recourse
to the same guessing approach as in \cite{BG08}.
It is easy to write down a recursion relation for 
 $\tilde{Y}_{s,t,u}$ of the same type as \eqref{eq:recurY}, obtained again by decomposing each 
 branch of an $(s,t,u)^+$-very-well-labelled Y-diagram at its first passage
at label $1$. It is in practice the same as \eqref{eq:recurY} with $T_s$ and $X_{s,t}$
replaced by their tilde counterparts, i.e.,
\begin{equation}
\tilde{Y}_{s,t,u}=1+g^3\, \tilde{T}_s \tilde{T}_t \tilde{T}_u \tilde{X}_{s+1,t+1} \tilde{X}_{s+1,u+1} \tilde{X}_{t+1,u+1}\tilde{T}_{s+1}\tilde{T}_{t+1}\tilde{T}_{u+1}\tilde{Y}_{s+1,t+1,u+1}
\label{eq:recurYtilde}
\end{equation}
and determines $\tilde{Y}_{s,t,u}$ entirely as a power series in $g$. We have been able to guess the 
solution of this equation, which has the slightly more involved expression (now a sum of two terms)
\begin{equation}
\begin{split}
& \hskip -10pt \tilde{Y}_{s,t,u}=\frac{[s+4]_x[t+4]_x[u+4]_x}{[3]_x[4]_x[s+2]_x[t+2]_x[u+2]_x[s+t+4]_x[t+u+4]_x[u+s+4]_x}\times \\
& \\
&  \hskip -2pt \times (x [3]_x[s\!+\!1]_x[t\!+\!1]_x[u\!+\!1]_x[s\!+\!t\!+\!u\!+\!5]_x\!+\![1]_x[s\!+\!3]_x[t\!+\!3]_x[u\!+\!3]_x[s\!+\!t\!+\!u\!+\!3]_x)\ .\\
\end{split}
\end{equation}
By the same argument as in Sect.~\ref{sec:genthreepoint}, we obtain directly:
\begin{prop}[Three-point function of bipartite maps]
Given $d_{12}$, $d_{13}$ and $d_{23}$ three positive integers satisfying strict triangular inequalities,
and such that $d_{12}+d_{13}+d_{23}$ is even, the three-point function $\tilde{G}_{d_{12},d_{13},d_{23}}$ 
is given by
\begin{equation}
\begin{split}
&\hskip -10pt \tilde{G}_{d_{12},d_{13},d_{23}}=\Delta_s\Delta_t\Delta_u \tilde{F}_{s,t,u}\\
& \hskip -10pt \tilde{F}_{s,t,u}=\tilde{X}_{s,t}\tilde{X}_{s,u}\tilde{X}_{t,u}\tilde{Y}_{s,t,u}^2\\
& \hskip -10pt = \frac{[4]_x(x [3]_x[s\!+\!1]_x[t\!+\!1]_x[u\!+\!1]_x[s\!+\!t\!+\!u\!+\!5]_x\!+\![1]_x[s\!+\!3]_x[t\!+\!3]_x[u\!+\!3]_x[s\!+\!t\!+\!u\!+\!3]_x)^2}{([2]_x)^3([3]_x)^2[s\!+\!t\!+\!2]_x[t\!+\!u\!+\!2]_x[u\!+\!s\!+\!2]_x[s\!+\!t\!+\!4]_x[t\!+\!u\!+\!4]_x[u\!+\!s\!+\!4]_x}\\ 
& \\
& \hskip 220pt \text{with}\ s,t,u\ \text{as in \eqref{eq:dstubip}}\ .
\end{split}
\label{eq:threepointbip}
\end{equation}
\end{prop}
As before, the case where triangular inequalities are not strict requires a special attention: setting again $u=0$ for instance, we now
arrive at
\begin{equation} 
\tilde{G}_{d_{13}+d_{23},d_{13},d_{23}} =\Delta_s\Delta_t \tilde{X}_{s,t}\quad \text{with}\ s=d_{13}, \  t=d_{23}\ .
\label{eq:bipaligned}
\end{equation}
Again this latter expression coincides with that given by \eqref{eq:threepointbip} if we formally set 
$\tilde{F}_{s,t,-1}=0$ so that $\Delta_u \tilde{F}_{s,t,u}\vert_{u=0}= \tilde{F}_{s,t,0}=\tilde{X}_{s,t}$. 

A simple application of \eqref{eq:threepointbip} is the small $g$ expansion
of $\tilde{G}_{2,2,2}$ (three vertices at pairwise distances $d_{12}=d_{13}=d_{23}=2$) obtained by setting
$s=t=u=1$ in \eqref{eq:threepointbip}. 
From the relation between $g$ and $x$ in \eqref{eq:2pbipTi}, we deduce the expansion
\begin{equation}
x=g+4 g^2+21 g^3+124 g^4+782 g^5+5144 g^6+34845 g^7+241196 g^8+\ldots
\end{equation}
Plugging this expansion in \eqref{eq:threepointbip} yields
\begin{equation}
\tilde{G}_{2,2,2} =2 g^3+21 g^4+174 g^5+1336 g^6+9942 g^7+72966 g^8 \ldots\\
\end{equation}
whose first terms (of order $g^3$ and $g^4$) are easily recovered by inspection.

\section{Bivariate two- and three-point functions}
\label{sec:bivgenfunc}
We can refine our analysis of the two- and three-point functions by keeping track of both the numbers
of edges and faces of the maps. More precisely, we may compute the {\it bivariate} two- and three-point
functions $G_{d_{12}}(g,z)$ and $G_{d_{12},d_{13},d_{23}}(g,z)$ for general planar maps (and their tilde analogs for
bipartite maps) enumerated with both a weight $g$ per
edge and a weight $z$ {\it per face} (and with a factor $1/k$ in case of $k$-fold symmetry). As we did before in the univariate case, 
we shall omit in the following the arguments $g$ and $z$ in all the encountered generating functions, and write for instance
$G_{d_{12}}$ and $G_{d_{12},d_{13},d_{23}}$ for short. All the generating functions discussed in this section are implicitly 
understood as bivariate generating functions, depending on both $g$ and $z$.

\subsection{General maps}
\label{sec:bivgentwothreepoint}
In the various bijections of Sect.~2, the faces of the bi- or tri-pointed maps at hand are in one-to-one
correspondence with local max of the associated well-labelled maps. Recall that a local max is a vertex
whose label is not smaller than that of any of its neighbours. Such a local max will now be assigned an additional
weight $z$.  Our first input is thus, for $s>0$, the generating function $T_s\equiv T_s(g,z)$ of
$(s)^+$-well-labelled maps, enumerated with a weight 
$g$ per edge and a weight $z$ per local max, and with a marked corner at a vertex (the root vertex) labelled $0$. 
As was done in \cite{AmBudd,BFG}, it is useful to also introduce the generating function $U_s\equiv U_s(g,z)$ of 
the same objects but {\it where the root vertex is weighted by $1$} irrespectively of whether or not this vertex is a local max.
Introducing the notation   
\begin{equation}
[s]_{x,\alpha} \equiv 1-\alpha\, x^s\ ,
\label{eq:crochetialpha}
\end{equation}
the following expressions for $T_s$ and $U_s$ were obtained in \cite{AmBudd,BFG}
\begin{equation}
\begin{split}
& T_s =T \frac{[s]_{x,1}[s+3]_{x,\alpha^2}}{[s+1]_{x,\alpha}[s+2]_{x,\alpha}}\ , \quad U_s =U \frac{[s]_{x,1}[s+3]_{x,\alpha}}{[s+1]_{x,1}[s+2]_{x,\alpha}} \ ,\\
& \text{where}\ \ T= \frac{\alpha (1-x)^2 (1+x+\alpha x-6
      \alpha\, x^2+\alpha\, x^3 + \alpha^2\, x^3 + \alpha^2\, x^4)}{(1-\alpha\, x )^3 (1-\alpha\, x^3)}\ , \\ 
& \hskip 32pt U= \frac{1+x+\alpha x-6
      \alpha\, x^2+\alpha\, x^3 + \alpha^2\, x^3 + \alpha^2\, x^4}{(1-\alpha\, x) (1-\alpha\, x^3)}\ , \\
&\hskip 32pt  g = \frac{x (1 - \alpha\, x)^3 (1 - \alpha\, x^3)}{ (1+x+\alpha x-6
      \alpha\, x^2+\alpha\, x^3 + \alpha^2\, x^3 + \alpha^2\, x^4)^2} \ , \\
&\hskip 32pt z = \frac{\alpha (1 - x)^3 (1 - \alpha^2\, x^3)}{
      (1 - \alpha\, x)^3 (1 - \alpha\, x^3)}\ . \\
\label{eq:TiUi}
\end{split}
\end{equation}
The bivariate two-point function was then computed, with result
 \begin{equation}
 G_{d_{12}}= \log\left(\frac{1+g\, U_{s}T_{s+1}}{1+g \, U_{s-1}T_{s}}\right)
 =\log\left(\frac{([s+1]_{x,\alpha})^3[s+3]_{x,\alpha}}{[s]_{x,\alpha}([s+2]_{x,\alpha})^3}\right)\quad \text{with}\ s=d_{12}\ .
 \label{eq:twopointbivariate}
 \end{equation}
Alternatively, the bijection (A) of Theorem \ref{theo:bij_bipointed} allows us to identify $G_{d_{12}}$ with $d_{12}=s+t$ 
with the generating function of $(s,t)$-well-labelled maps of type A,
enumerated with a weight $g$ per edge and $z$ per local max. Introducing the bivariate generating function $N_{s,t}\equiv N_{s,t}(g,z)$ of 
$(s,t)^+$-well-labelled chains of type A (with the convention that the configuration reduced to a single spine vertex receives the weight $1$), 
we note that, as before, $(s,t)^+$-well-labelled maps of type A are simply enumerated by $\log(N_{s,t})$. Indeed, $(s,t)^+$-well-labelled chains of type A and $(s,t)^+$-well-labelled maps of type A correspond again to sequences and cyclic sequences of the same elementary blocks, corresponding to 
$(s,t)^+$-well-labelled chains of type A having no spine label $0$ between the extremities of their spine, now enumerated with the additional weight $z$ per local max. Here it is 
crucial to realize that, upon gluing (linearly or cyclically) the blocks,  {\it the} (local max or not) {\it nature of the gluing vertices is not affected}
since these vertices, with label $0$, are never local max within objects of type A.
This allows us to write as before
\begin{equation}
G_{d_{12}}=\Delta_s\Delta_t \log(N_{s,t})=\log\left(\frac{N_{s,t}N_{s-1,t-1}}{N_{s-1,t}N_{s,t-1}}\right)\quad \text{with}\ s+t=d_{12}\ ,
\label{eq:twopointbisbivariate}
\end{equation}
(with the convention $N_{s,0}=N_{0,t}=N_{0,0}=1$) and, by comparing with the expression \eqref{eq:twopointbivariate}, to obtain the formula
\begin{equation}
N_{s,t}=\frac{[3]_{x,\alpha}[s+2]_{x,\alpha}[t+2]_{x,\alpha}[s+t+3]_{x,\alpha}}{[2]_{x,\alpha}[s+3]_{x,\alpha}[t+3]_{x,\alpha}[s+t+2]_{x,\alpha}}\ .
\label{eq:Nstbivariate}
\end{equation}

\begin{figure}[h!]
\begin{center}
\includegraphics[width=10cm]{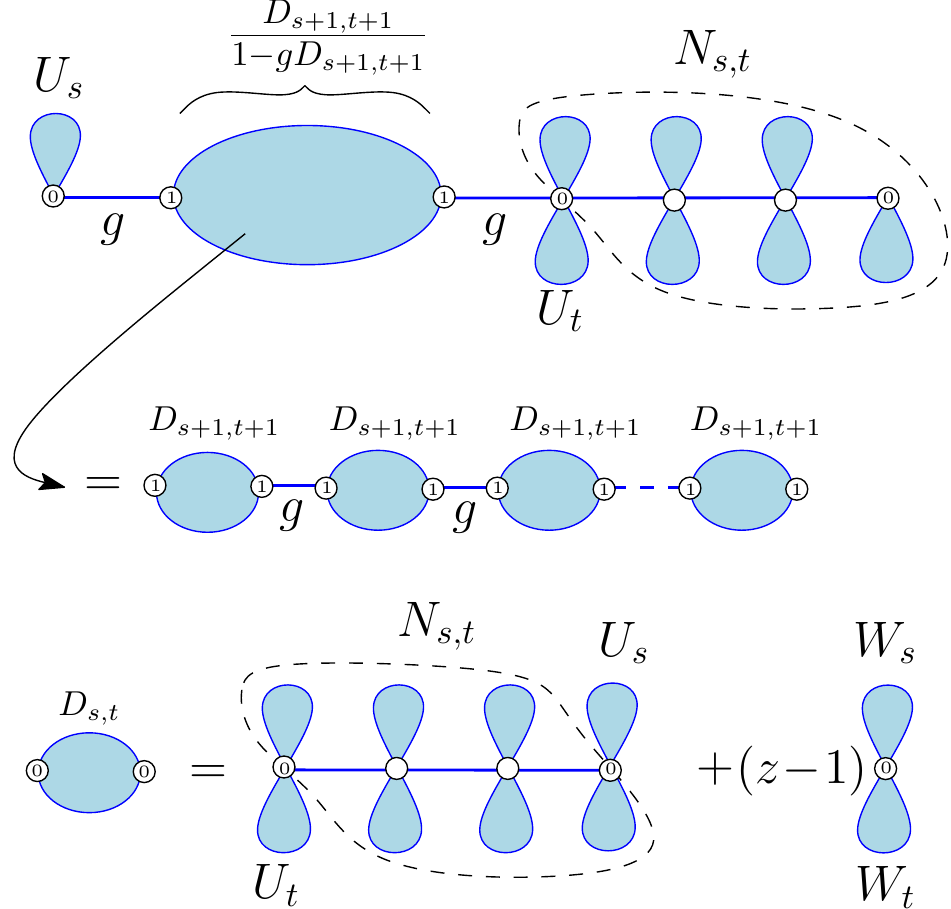}
\end{center}
\caption{Recursion relation for $N_{s,t}$ obtained by decomposing an $(s,t)^+$-well-labelled chain of type A
at its first return at label $0$ along the spine (see text).} 
\label{fig:nstdecompbiv}
\end{figure}

As a test of consistency, we may write down a recursion relation for $N_{s,t}$, the bivariate analog of \eqref{eq:recurN}, illustrated
in Figure~\ref{fig:nstdecompbiv}. It now reads
\begin{equation}
N_{s,t}=1+g^2 U_s U_t N_{s,t} \frac{D_{s+1,t+1}}{1-g\, D_{s+1,t+1}}
\label{eq:recurNbivariate}
\end{equation}
where 
\begin{equation}
D_{s,t}=U_{s}U_{t}N_{s,t}+(z-1)W_sW_t\ , \qquad W_s=\frac{U_s}{1+g\, U_s T_{s+1}}\ .
\label{eq:DstWs}
\end{equation}
Indeed, when not reduced to a single spine vertex (weight $1$), an $(s,t)^+$-well-labelled chain of type A has
a spine made of a first $0-1$ edge (weight $g\,U_s$ with the attached subtree since the vertex labelled $0$ on the spine is not a local max), 
a portion of spine with labels larger than or equal to $1$, then a first $1-0$ edge (weight $g\, U_t$ since the vertex labelled $0$ 
is not a local max) and a final portion which is itself an $(s,t)^+$-well-labelled chain of type A
(weight $N_{s,t}$). The portion of spine with labels larger than or equal to $1$ is, after a simple shift of labels by $-1$,
an $(s+1,t+1)^+$-well-labelled chain of type A or B with two extra attached subtrees. By a simple decomposition of
this chain at each edge with labels $1-1$, it is now enumerated by $D_{s+1,t+1}/(1-g\, D_{s+1,t+1})$ where $D_{s,t}$
enumerates $(s,t)^+$-well-labelled chains of type A with two extra attached trees, with the slight modification that,
when the chain is reduced to a single spine vertex, and when the two extra attached trees are such that
this spine vertex is a local max, then the spine vertex should receive a weight $z$ instead of $1$. 
This leads to the above expression \eqref{eq:DstWs} for $D_{s,t}$ with a first term $U_{s}U_{t}N_{s,t}$
giving a weight $1$ to the chain reduced to a single spine vertex, then corrected by a term $(z-1)W_sW_t$
accounting for the case where the single spine vertex in a local max. Here $W_s$ enumerates
$(s)^+$-well-labelled maps with a marked corner at a vertex (the root vertex) having label $0$ and {\it being a local max}.
By a simple canonical decomposition of an arbitrary $(s)^+$-well-labelled map with a marked corner at a vertex labelled $0$ 
(this map is a planted tree enumerated by $U_s$) by marking each of its descending subtrees with root label $1$ (each
such descending subtree is enumerated by $gT_{s+1}$ and the part between two such subtrees is enumerated by $W_s$), we have 
$U_s=W_s/(1-g\, T_{s+1}W_s)$, hence the expression \eqref{eq:DstWs}  for $W_s$. 

Using the explicit forms \eqref{eq:TiUi}, we arrive at 
\begin{equation}
g\, D_{s,t}=\frac{\alpha\, x\, [1]_{x,1}[s]_{x,1}[t]_{x,1}[s+t+3]_{x,\alpha^2}}{[2]_{x,\alpha}[s+1]_{x,\alpha}[t+1]_{x,\alpha}[s+t+2]_{x,\alpha}}
\end{equation}
and it is a straightforward exercise to check that \eqref{eq:Nstbivariate} actually solves \eqref{eq:recurNbivariate}.

We may alternatively evaluate the two-point function by using instead the bijection (B) of Theorem \ref{theo:bij_bipointed}, writing now
$d_{12}= s+t-1$ and identifying $G_{d_{12}}$ with the bivariate generating function of $(s,t)$-well-labelled maps of type B
enumerated with a weight $g$ per edge and $z$ per local max. As before, $(s,t)^+$-well-labelled maps of type 
B are cyclic sequences where each elementary block
consists of an edge with labels $0-0$ (weight $g$) followed by an $(s,t)^+$-well-labelled chain of type A with two 
extra attached subtrees, as enumerated by $D_{s,t}$ (again if the chain is reduced to a single vertex and 
the two extra attached trees are such that this vertex is a local max, then it should receive the weight $z$).
This leads immediately to
\begin{equation}
\begin{split}
G_{d_{12}}& =\Delta_s\Delta_t \log\left(\frac{1}{1-g\, D_{s,t}}\right) \\
&=\Delta_s\Delta_t \log\left(\frac{[2]_{x,\alpha}[s+1]_{x,\alpha}[t+1]_{x,\alpha}[s+t+2]_{x,\alpha}}{[1]_{x,\alpha}
[s+2]_{x,\alpha}[t+2]_{x,\alpha}[s+t+1]_{x,\alpha}}\right) \\
&= \log\left(\frac{([s+t]_{x,\alpha})^3[s+t+2]_{x,\alpha}}{[s+t-1]_{x,\alpha}([s+t+1]_{x,\alpha})^3}\right)\quad \text{with}\ s+t-1=d_{12}\ . \\
\end{split}
\label{eq:twopointterbivariate}
\end{equation}
in agreement with \eqref{eq:twopointbivariate}, as wanted.

Let us now compute the bivariate three-point function $G_{d_{12},d_{13},d_{23}}=G_{d_{12},d_{13},d_{23}}(g,z)$,
starting with the case where $d_{12}+d_{13}+d_{23}$ is even (and $d_{12}$, $d_{13}$ and $d_{23}$ satisfy strict triangular inequalities), using the parametrization \eqref{eq:dstupair}.
In this case, we have as before
\begin{equation}
G_{d_{12},d_{13},d_{23}}=\Delta_s\Delta_t\Delta_u N_{s,t}N_{s,u}N_{t,u}Y_{s,t,u}^2\ ,
\label{eq:threepointfirstbivariate}
\end{equation}
with $N_{s,t}$ as in \eqref{eq:Nstbivariate} and where $Y_{s,t,u}$ is now the bivariate generating function for 
$(s,t,u)^+$-well-labelled Y-diagrams (here again we use the convention that the Y-diagram reduced to a single backbone
vertex receives the weight $1$).  Once more, it is 
crucial to realize that, upon gluing the different pieces to get an $(s,t,u)^+$-well-labelled  map of type A, the (local max or not) nature of the
gluing vertices is not affected since these vertices, with label $0$, are never local max.
To evaluate $Y_{s,t,u}$, we can write a recursion relation analog to \eqref{eq:recurY}, obtained by decomposing each branch of the 
Y-diagram at its first passage at label $1$ (from the central vertex) along the branch. It reads
\begin{equation}
\begin{split}
&\hskip -10pt Y_{s,t,u}=1+g^3\, U_s U_t U_u \frac{1}{1-g\, D_{s+1,t+1}}\frac{1}{1-g\, D_{s+1,u+1}}\frac{1}{1-g\, D_{t+1,u+1}}\times\\
& \hskip 20pt \times (N_{s+1,t+1} N_{s+1,u+1} N_{t+1,u+1}U_{s+1}U_{t+1}U_{u+1}Y_{s+1,t+1,u+1} \\
& \hskip 170 pt
+(z-1)W_{s+1}W_{t+1}W_{u+1})\ ,\\
\label{eq:recurYbivariate}
\end{split}
\end{equation}
with the following interpretation: going towards the central vertex, each branch is formed of a first $0-1$ edge
(weight $g U_s$, $g\, U_t$ and $g\, U_u$ respectively for the three branches, the three vertices with label
$0$ being not local max) followed by a portion of backbone until we reach the last $1-1$ edge on the branch
(if there is no $1-1$ edge on the branch, this portion is empty).
This portion is enumerated by $1/(1-g\, D_{s+1,t+1})$, $1/(1-g\, D_{s+1,u+1})$ and $1/(1-g\, D_{t+1,u+1})$ respectively on the
three branches. After this last $1-1$ edge (or after the first $0-1$ edge if the portion is empty), the branch continues with a portion without $1-1$ edges,
as enumerated by $U_{t+1}N_{s+1,t+1}$, $U_{u+1}N_{t+1,u+1}$ and $U_{s+1}N_{s+1,u+1}$ respectively until the
vertex with label $1$ closest to the central vertex is reached (note the presence of the terms $U_{t+1}$, $U_{u+1}$ or $U_{s+1}$
accounting for the tree attached to the right of the first vertex in this portion). The remaining part is (by shifting the labels by $-1$)
an $(s+1,t+1,u+1)^+$-well-labelled Y-diagram, enumerated by $Y_{s+1,t+1,u+1}$, hence the first term in the parentheses in \eqref{eq:recurYbivariate}. 
In the above enumeration, we have assumed
that the extremities of the last portion without $1-1$ edge were not local max. This is true except when
the central vertex itself has label $1$, when the last portion without $1-1$ edge is of length $0$ for each branch,
and when the trees attached to the central vertex are such that this vertex has no neighbours with larger labels.
This explains the correction $(z-1)W_{s+1}W_{t+1}W_{u+1}$ in the parentheses.

Although equation \eqref{eq:recurYbivariate} may appear slightly involved, its solution is, remarkably, the simplest possible generalization of
\eqref{eq:Ystu} that we may think of, namely
\begin{equation}
Y_{s,t,u}=\frac{[s+3]_{x,\alpha}[t+3]_{x,\alpha}[u+3]_{x,\alpha}[s+t+u+3]_{x,\alpha}}{[3]_{x,\alpha}[s+t+3]_{x,\alpha}[t+u+3]_{x,\alpha}[u+s+3]_{x,\alpha}}\ .
\end{equation}
Plugging this expression in \eqref{eq:threepointfirstbivariate}, we obtain:
\begin{prop}[Bivariate three-point function of general maps: even case]
Given $d_{12}$, $d_{13}$ and $d_{23}$ three positive integers satisfying strict triangular inequalities,
and such that $d_{12}+d_{13}+d_{23}$ is even, the bivariate three-point function $G_{d_{12},d_{13},d_{23}}$ 
is given by
\begin{equation}
\begin{split}
& \hskip -20pt G_{d_{12},d_{13},d_{23}} =\Delta_s\Delta_t\Delta_u  F^{\text{even}}_{s,t,u}\\
& \hskip -20pt F^{\text{even}}_{s,t,u}=\ \frac{[3]_{x,\alpha}([s\!+\!2]_{x,\alpha}[t\!+\!2]_{x,\alpha}[u\!+\!2]_{x,\alpha}[s\!+\!t\!+\!u\!+\!3]_{x,\alpha})^2}{([2]_{x,\alpha})^3[s\!+\!t\!+\!2]_{x,\alpha}[t\!+\!u\!+\!2]_{x,\alpha}[u\!+\!s\!+\!2]_{x,\alpha}
[s\!+\!t\!+\!3]_{x,\alpha}[t\!+\!u\!+\!3]_{x,\alpha}[u\!+\!s\!+\!3]_{x,\alpha}}\\
& \\
& \hskip 220pt \text{with}\ s,t,u\ \text{as in \eqref{eq:dstupair}}\ .\\
\label{eq:threepointgenpairbivariate}
\end{split}
\end{equation} 
\end{prop}
This formula deals with situations where the $d_{12}$, $d_{23}$ and $d_{31}$ satisfy strict triangular inequalities. If not,
say $d_{13}+d_{23}=d_{12}$, 
one may easily verify that the relation $G_{d_{13}+d_{23},d_{13},d_{23}}=\Delta_s\Delta_t N_{s,t}$ (with $d_{13}=s$ and $d_{23}=t$) 
remains valid in the bivariate case, now with $N_{s,t}$ given by \eqref{eq:Nstbivariate}.

The case where $d_{12}+d_{13}+d_{23}$ is odd requires the bivariate generating function of $(s,t,u)$-well-labelled maps 
of type B, with $s$, $t$ and $u$ as in \eqref{eq:dstuimpair}. The generating function of $(s,t,u)^+$-well-labelled maps 
of type B is easily obtained by decomposing these maps into five pieces upon cutting their backbone at the first and last occurrence 
of an edge with labels $0-0$ along each border between their three faces. This yields
\begin{equation}
\hskip -15pt F_{s,t,u}^{\text{odd}}=g^3\, \frac{1}{1\!-\!g\, D_{s,t}}\frac{1}{1\!-\!g\, D_{s,u}}\frac{1}{1\!-\!g\, D_{t,u}}(N_{s,t} N_{s,u} N_{t,u}U_{s}U_{t}U_{u}Y_{s,t,u}\!+\!(z\!-\!1)W_{s}W_{t}W_{u})^2\ .
\end{equation}
We leave the proof of this formula to the reader, who will recognize the same basic building blocks as in the derivation of \eqref{eq:recurYbivariate}.
This leads to:
\begin{prop}[Bivariate three-point function of general maps: odd case]
Given $d_{12}$, $d_{13}$ and $d_{23}$ three positive integers satisfying triangular inequalities,
and such that $d_{12}+d_{13}+d_{23}$ is odd, the bivariate three-point function $G_{d_{12},d_{13},d_{23}}$ 
is given by
\begin{equation}
\begin{split}
& \hskip -20pt G_{d_{12},d_{13},d_{23}} =\Delta_s\Delta_t\Delta_u  F^{\text{odd}}_{s,t,u}\\
& \hskip -20pt F^{\text{odd}}_{s,t,u}=x^3 \frac{[3]_{x,\alpha}(\alpha\, [s]_{x,1}[t]_{x,1}[u]_{x,1}[s\!+\!t\!+\!u\!+\!3]_{x,\alpha^2})^2}{([2]_{x,\alpha})^3[s\!+\!t\!+\!1]_{x,\alpha}[t\!+\!u\!+\!1]_{x,\alpha}
[u\!+\!s\!+\!1]_{x,\alpha}
[s\!+\!t\!+\!2]_{x,\alpha}[t\!+\!u\!+\!2]_{x,\alpha}[u\!+\!s\!+\!2]_{x,\alpha}}\\
& \\
& \hskip 220pt \text{with}\ s,t,u\ \text{as in \eqref{eq:dstuimpair}}\ .\\
\label{eq:threepointgenimpairbivariate}
\end{split}
\end{equation} 
\end{prop}

Again, as a simple application of our formulas, let us revisit the small $g$ expansion
of $G_{2,2,2}$  and $G_{1,1,1}$ .
From the relation between $g$, $z$, $x$ and $\alpha$ in \eqref{eq:TiUi}, we deduce the expansions
\begin{equation}
\begin{split}
& \hskip -10pt x= g\!+\!(2\!+\!5 z) g^2\!+\!(5\!+\!31z\!+\!23z^2) g^3\!+\!(14\!+\!153z\!+\!275 z^2\!+\!102z^3) g^4\!+\\
& \hskip 20pt\!+\!
(42\!+\!696z\!+\!2170z^2\!+\!1938z^3\!+\!443 z^4)g^5\\
&\hskip 20pt \!+\!(132\!+\!3042z\!+\!14212 z^2\!+\!21937 z^3\!+\!12035 z^4\!+\!1898 z^5)g^6+
\!\ldots \\
& \\
& \hskip -10pt \alpha= z\!+\!3z(1\!-\!z) g\!+\!3 z(1\!-\!z)(4\!+\!z) g^2\!+\!z(1\!-\!z)(49\!+\!51z \!+\!4 z^4) g^3 \\
&\hskip 20pt\!+\!3z(1\!-\!z)(67\!+\!150z\!+\!62z^2\!+\!2z^3) g^4\\
&\hskip 20pt\!+\!3z(1\!-\!z)(275\!+\!1038z\!+\!955z^2\!+\!219z^3\!+\!3z^4)g^5\\
&\hskip 20pt\!+\!z(1\!-\!z)(3384\!+\!18965z\!+\!29747z^2\!+\!15651z^3\!+\!2310z^4\!+\!11z^5)g^6
\!+\!\ldots \\
\end{split}
\end{equation}
and consequently
\begin{equation}
\begin{split}
&\hskip -30pt G_{2,2,2} = 2 z g^3\!+\!3 z(4\!+\!9z) g^4\!+\!18z(3\!+\!15z\!+\!13z^2)g^5\!+\!z(220\!+\!1795z\!+\!3453 z^2\!+\!1655 z^3) g^6\!+\!\ldots\\
& \\
&\hskip -30pt G_{1,1,1} = z^2 g^3\!+\!3 z^2(2\!+\!3z) g^4\!+\!3z^2(9\!+\!30z\!+\!19z^2) g^5\!+\!z^2(110\!+\!600z\!+\!845 z^2\!+\!312 z^3)g^6\!+\!\ldots \\
\end{split}
\end{equation}
whose first terms (of order $g^3$ and $g^4$) may be checked by a simple inspection.

\begin{figure}[h!]
\begin{center}
\includegraphics[width=5cm]{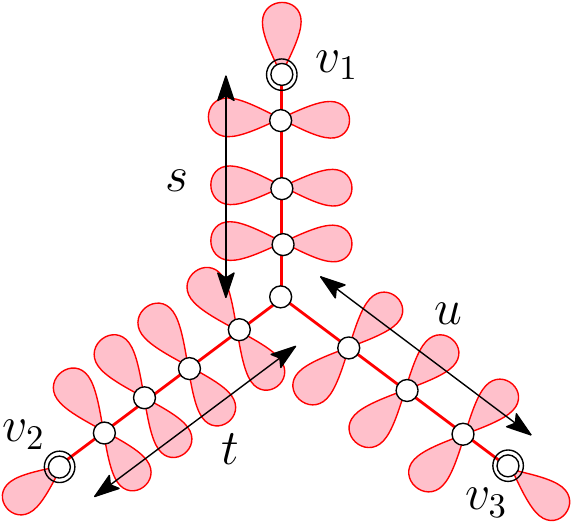}
\end{center}
\caption{Schematic picture of a tri-pointed tree with pairwise distances $d_{12}=s+t$, $d_{13}=s+u$ and $d_{23}=t+u$. The blobs
represent attached subtrees.} 
\label{fig:treeconfig}
\end{figure}

It is interesting to look at the $z\to 0$ limit of our three-point function. From \eqref{eq:TiUi},
this limit is reached by letting $\alpha \to 0$, in which case
\begin{equation}
g=\frac{x}{(1+x)^2}+O(\alpha)\ , \qquad z=(1-x)^3 \alpha +O(\alpha^2)\ .
\end{equation} 
Expanding $F_{s,t,u}^{\text{even}}$ at first order in $\alpha$, we find
\begin{equation}
\hskip -10pt F_{s,t,u}^{\text{even}}=\frac{x^2 \left(3\!-\!x\!-\!2x^s\!-\!2x^t\!-\!2x^u\!+(1\!+\!x)(x^{s+t}\!+\!x^{t+u}\!+\!x^{u+s})-\!2x^{s+t+u+1}\right)}{(1-x)^3}
\, z+O(z^2)
\end{equation}
so that, from \eqref{eq:threepointgenpairbivariate}, in the case where $d_{12}+d_{13}+d_{23}$ is even
\begin{equation}
\begin{split}
& G_{d_{12},d_{13},d_{23}} =2\, x^{s+t+u}\, z+O(z^2)=2\, x^{\frac{d_{12}+d_{13}+d_{23}}{2} }\, z+O(z^2)\\
&\text{with}\ \  x= g\, \left( \frac{1-\sqrt{1-4g}}{2g}\right)^2\ .\\
\label{eq:threepointzzeropair}
\end{split}
\end{equation} 
This corresponds, as it should, to the {\it three-point function of planar trees} enumerated with a weight $g$ per edge.
It enumerates tri-pointed trees of the form displayed in Figure~\ref{fig:treeconfig}, with a backbone made of three branches of 
respective lengths $s$, $t$ and $u$ as in \eqref{eq:dstupair}. A weight $g$ is attached to each backbone edge
and a weight $\text{Cat}(g)=(1-\sqrt{1-4g})/(2g)$ (the Catalan function) to each backbone corner, to account for the possible
subtrees attached to the backbone. This leads to the above formula since the backbone has $s+t+u$ edges
and $2(s+t+u)$ corners. The prefactor $2$ in \eqref{eq:threepointzzeropair} comes from the two possible (clockwise or counterclockwise) 
cyclic orders in which the three distinguished vertices $v_1$, $v_2$ and $v_3$ may appear in the plane.

\begin{figure}[h!]
\begin{center}
\includegraphics[width=6.cm]{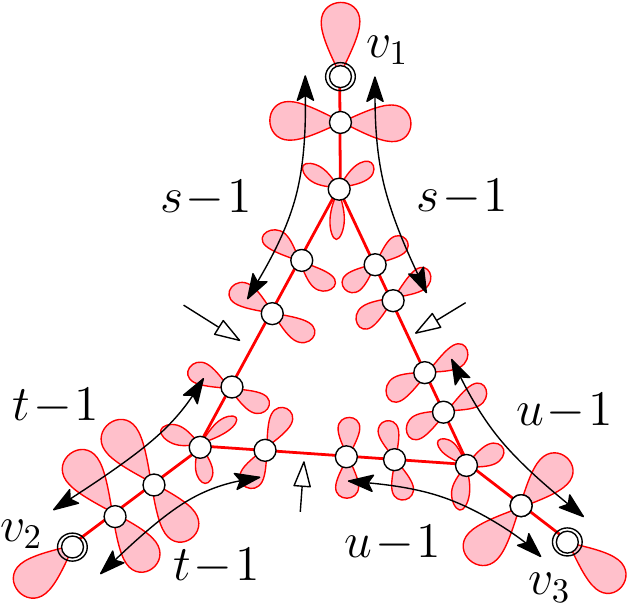}
\end{center}
\caption{Schematic picture of a tri-pointed map with two faces and pairwise distances $d_{12}=s+t-1$, $d_{13}=s+u-1$ and $d_{23}=t+u-1$.
The blobs represent attached subtrees. Suppressing the three edges marked by an open arrow cuts the map into three parts, corresponding 
to those backbone edges and vertices at distance $\leq s-1$ from vertex $v_1$ (resp.\ $\leq t-1$ from $v_2$,  $\leq u-1$ from $v_3$). 
Each of these parts is partially unzipped, the zipped portion lying either inside or outside the open triangular part.} 
\label{fig:triangleconfig}
\end{figure}

If $d_{12}+d_{13}+d_{23}$ is odd, we obtain from
\begin{equation}
F_{s,t,u}^{\text{odd}}=\frac{x^3 \left((1-x^s)(1-x^t)(1-x^u)\right)^2}{(1-x)^6}
\, z^2+O(z^3)
\end{equation}
the leading order in $z$
\begin{equation}
\begin{split}
& \hskip -20pt G_{d_{12},d_{13},d_{23}} =x^{s+t+u}\frac{(2-x^s-x^{s-1})(2-x^t-x^{t-1})(2-x^u-x^{u-1})}{(1-x)^3}\, z^2+O(z^3)\\
& \text{with}\ \  x= g\, \left( \frac{1-\sqrt{1-4g}}{2g}\right)^2\ \ \text{and}\ \ s,t,u\ \text{as in \eqref{eq:dstuimpair}}\ .\\
\label{eq:threepointzzeroimpair}
\end{split}
\end{equation} 
This now corresponds, as it should, to the enumeration of configurations of tri-pointed maps having two faces, 
which is the minimal number of faces in this case since the map cannot be bipartite. As displayed in Figure~\ref{fig:triangleconfig},
such configurations are characterized by a backbone which forms a ``geodesic triangle" between the three marked vertices and is made
of an open triangular part with three attached ``legs". The backbone is then completed by subtrees
attached to all its corners and, as in the even case, the number of corners of the backbone is twice its number
of edges so the correct enumeration is performed by simply assigning a weight $x=g\, \text{Cat}^2(g)$ to each backbone edge. 
The lengths of the different portions of the backbone are constrained by the pairwise distances and we may distinguish three parts:
a part made of those backbone-edges and vertices at distance less than or equal to $s-1$ from vertex $v_1$,
a part made of those backbone-edges and vertices at distance less than or equal to $t-1$ from vertex $v_2$,
a part made of those backbone-edges and vertices at distance less than or equal to $u-1$ from vertex $v_3$.
The backbone is then the union of these three parts and of three remaining edges. Each part is partially unzipped, the zipped portion 
corresponding to one of the legs of the backbone. This allows us to write the generating function as $x^3 Z_s Z_t Z_u$ where
$Z_s$ properly enumerates the first part. If we call $i$ the length of its zipped portion, we find 
\begin{equation}
Z_s=\sum_{i=0}^{s-1} (2-\delta_{i,0}) x^{i+2(s-1-i)}=x^{s-1}\frac{2-x^s-x^{s-1}}{1-x}\ .
\end{equation}
Note the factor of $2$ whenever $i>0$. Indeed, having distinguished the three marked vertices, 
we may canonically differentiate between the two faces in the map and call one, say the interior
and the other the exterior. Whenever the zipped portion has a non-zero length, we then must decide 
whether it lies in the interior or the exterior. There is no such choice for $i=0$. 
With this expression of $Z_s$, we recover the formula \eqref{eq:threepointzzeroimpair}.

\subsection{Bipartite maps}
\label{sec:bivbiptwothreepoint}
Let us now come to the bivariate two- and three-point functions $\tilde{G}_{d_{12}}(g,z)$ and $\tilde{G}_{d_{12},d_{13},d_{23}}(g,z)$
of bipartite planar maps. As a first building block, we consider, for $s>0$, the bivariate generating function 
$\tilde{T}_s\equiv \tilde{T}_s(g,z)$ of $(s)^+$-very-well-labelled maps, enumerated with a weight 
$g$ per edge, $z$ per local max, and with a marked corner at a vertex labelled $0$. As before, we also introduce
the generating function $\tilde{U}_s=\tilde{U}_s(g,z)$ of the same objects, but with a weight $1$ for the root vertex even
if it is a local max. Expressions for $\tilde{T}_s$ and $\tilde{U}_s$ were given in \cite{BFG}
\begin{equation}
\begin{split}
& \tilde{T_s}=\tilde{T} \frac{[s]_{x,1}[s+4]_{x,\alpha^2}}{[s+1]_{x,\alpha}[s+3]_{x,\alpha}}\ , \quad \tilde{U_s}=\tilde{U} \frac{[s]_{x,1}[s+4]_{x,\alpha}}{[s+1]_{x,1}[s+3]_{x,\alpha}} \ ,\\
& \text{where}\ \ \tilde{T}= \frac{\alpha (1-x^2)^2 (1-\alpha\, x^2)}{(1-\alpha\, x )^2 (1-\alpha\, x^4)}\ , \\ 
& \hskip 32pt \tilde{U}= \frac{(1+x)\, (1-\alpha\, x^2)^2}{(1-\alpha\, x) (1-\alpha\, x^4)}\ , \\
&\hskip 32pt  g = \frac{x (1 - \alpha\, x)^2 (1 - \alpha\, x^4)}{ (1+x)^2(1-\alpha\, x^2)^3} \ , \\
&\hskip 32pt z = \frac{\alpha (1 - x)^2 (1-x^2)(1 + \alpha\, x^2)}{(1 - \alpha\, x)^2 (1 - \alpha\, x^4)}\ . \\
\label{eq:TiUitilde}
\end{split}
\end{equation}
As explained in \cite{BFG},  the two-point function of bipartite planar maps is then
\begin{equation}
 \tilde{G}_{d_{12}}= \log\left(\frac{1+g\, \tilde{U}_{s}\tilde{T}_{s+1}}{1+g \, \tilde{U}_{s-1}\tilde{T}_{s}}\right)
 =\log\left(\frac{([s+1]_{x.\alpha})^2[s+4]_{x,\alpha}}{[s]_{x,\alpha}([s+3]_{x,\alpha})^2}\right)\quad \text{with}\ s=d_{12}\ .
 \label{eq:twopointbipbivariate}
 \end{equation}
This formula may be recovered via Theorem \ref{theo:bij_bipointed}, as $\tilde{G}_{d_{12}}$ is also, for any pair of positive $s$ and $t$
with $s+t=d_{12}$, the generating function of $(s,t)$-very-well-labelled maps of type A, enumerated with a weight $g$ per edge and $z$ per local max. 
Consider the bivariate generating function $\tilde{X}_{s,t}\equiv \tilde{X}_{s,t}(g,z)$
of $(s,t)^+$- very-well-labelled chains, we now have again the relation \eqref{eq:twopointbipbis}, which leads to the identification
\begin{equation}
\tilde{X}_{s,t}=\frac{[4]_{x,\alpha}[s+2]_{x,\alpha}[t+2]_{x,\alpha}[s+t+4]_{x,\alpha}}{[2]_{x,\alpha}[s+4]_{x,\alpha}[t+4]_{x,\alpha}[s+t+2]_{x,\alpha}}\ .
\label{eq:tildeXstbivariate}
\end{equation}
As easily checked, $\tilde{X}_{s,t}$ now solves the following recursion relation
\begin{equation}
\tilde{X}_{s,t}=1+g^2\, \tilde{U}_s \tilde{U}_t \tilde{X}_{s,t}\left(\tilde{U}_{t+1}\tilde{U}_{s+1} \tilde{X}_{s_1,t+1}+(z-1) \tilde{W}_s \tilde{W}_t\right)\ ,
\end{equation}
with
\begin{equation} 
\tilde{W}_s=\frac{\tilde{U}_s}{1+g\, \tilde{U}_s \tilde{T}_{s+1}}\ .
\end{equation}

The bivariate generating function $\tilde{Y}_{s,t,u}\equiv \tilde{Y}_{s,t,u}(g,z)$ of very-well-labelled 
$(s,t,u)^+$-Y-diagrams, with a weight $g$ per edge and $z$ per local max, may be  
obtained by solving the bivariate analog of \eqref{eq:recurYtilde}. It now reads
\begin{equation}
\begin{split}
&\hskip -10pt \tilde{Y}_{s,t,u}=1+g^3\, \tilde{U}_s \tilde{U}_t \tilde{U}_u \big(\tilde{X}_{s+1,t+1} \tilde{X}_{s+1,u+1} \tilde{X}_{t+1,u+1}\tilde{U}_{s+1}\tilde{U}_{t+1}\tilde{U}_{u+1}\tilde{Y}_{s+1,t+1,u+1}\\
& \hskip 210pt +(z-1) \tilde{W}_{s+1}\tilde{W}_{t+1}\tilde{W}_{u+1}\big)\ .\\
\label{eq:recurYtildebivariate}
\end{split}
\end{equation}
Again we were able to guess the (slightly involved) solution of this equation
\begin{equation}
\begin{split}
& \hskip -30pt \tilde{Y}_{s,t,u}=\frac{[s+4]_{x,\alpha}[t+4]_{x,\alpha}[u+4]_{x,\alpha}}{[3]_{x,\alpha}[4]_{x,\alpha}[s+2]_{x,\alpha}[t+2]_{x,\alpha}
[u+2]_{x,\alpha}[s+t+4]_{x,\alpha}[t+u+4]_{x,\alpha}[u+s+4]_{x,\alpha}}\times \\
& \\
&  \hskip -10pt \times (\alpha\, x [3]_{x,\alpha}[s\!+\!1]_{x,1}[t\!+\!1]_{x,1}[u\!+\!1]_{x,1}[s\!+\!t\!+\!u\!+\!5]_{x,\alpha^2}\!\\
& \hskip 140pt +\![1]_{x,\alpha}
[s\!+\!3]_{x,\alpha}[t\!+\!3]_{x,\alpha}[u\!+\!3]_{x,\alpha}[s\!+\!t\!+\!u\!+\!3]_{x,\alpha})\ .\\
\end{split}
\end{equation}
This yields
\begin{prop}[Bivariate three-point function of bipartite maps]
Given $d_{12}$, $d_{13}$ and $d_{23}$ three positive integers satisfying strict triangular inequalities,
and such that $d_{12}+d_{13}+d_{23}$ is even, the three-point function $\tilde{G}_{d_{12},d_{13},d_{23}}$ 
is given by
\begin{equation}
\begin{split}
&\hskip -26pt \tilde{G}_{d_{12},d_{13},d_{23}}=\Delta_s\Delta_t\Delta_u \tilde{F}_{s,t,u}\\
& \hskip -26pt \tilde{F}_{s,t,u}=\tilde{X}_{s,t}\tilde{X}_{s,u}\tilde{X}_{t,u}\tilde{Y}_{s,t,u}^2\\
& \hskip -26pt = \frac{[4]_{x,\alpha}}{([2]_{x,\alpha})^3([3]_{x,\alpha})^2
[s\!+\!t\!+\!2]_{x,\alpha}[t\!+\!u\!+\!2]_{x,\alpha}[u\!+\!s\!+\!2]_{x,\alpha}[s\!+\!t\!+\!4]_{x,\alpha}[t\!+\!u\!+\!4]_{x,\alpha}[u\!+\!s\!+\!4]_{x,\alpha}}\times \\ 
&\\
&  \hskip -2pt \times (\alpha\, x [3]_{x,\alpha}[s\!+\!1]_{x,1}[t\!+\!1]_{x,1}[u\!+\!1]_{x,1}[s\!+\!t\!+\!u\!+\!5]_{x,\alpha^2}\!\\
& \hskip 150pt +\![1]_{x,\alpha}
[s\!+\!3]_{x,\alpha}[t\!+\!3]_{x,\alpha}[u\!+\!3]_{x,\alpha}[s\!+\!t\!+\!u\!+\!3]_{x,\alpha})^2\\
& \\
& \hskip 220pt \text{with}\ s,t,u\ \text{as in \eqref{eq:dstubip}}\ .
\end{split}
\label{eq:threepointbipbivariate}
\end{equation}
\end{prop}
In a situation where, say $u=0$, $\tilde{G}_{d_{12},d_{13},d_{23}}$ is still obtained via \eqref{eq:bipaligned},
now with the bivariate $\tilde{X}_{s,t}$.

From \eqref{eq:threepointbipbivariate}, we can get for instance the small $g$ expansion
of $\tilde{G}_{2,2,2}$ (three vertices at pairwise distances $d_{12}=d_{13}=d_{23}=2$) by setting
$s=t=u=1$ in \eqref{eq:threepointbipbivariate}. 
From \eqref{eq:TiUitilde}, we have the expansion
\begin{equation}
\begin{split}
& \hskip -20pt x= g\!+\!2(1\!+\!z) g^2\!+\!(5\!+\!13z\!+\!3z^2) g^3\!+\!(14\!+\!66z\!+\!40z^2\!+\!4z^3) g^4\!+\\
& \!+\!(42\!+\!306z\!+\!339z^2\!+\!90z^3\!+\!5 z^4)g^5
\!+\!2(66\!+\!678z\!+\!1168 z^2\!+\!572 z^3\!+\!85 z^4\!+\!3 z^5)g^6+
\!\ldots \\
& \\
& \hskip -20pt \alpha= z\!+\!2z(1\!-\!z) g\!+\!z(1\!-\!z)(8\!-\!z) g^2\!+\!32 z(1\!-\!z) g^3 
+3z(1\!-\!z)(43\!+\!14z\!) g^4\\
&\!+\!2z(1\!-\!z)(261\!+\!214z\!+\!26z^2)g^5\!+\!z(1\!-\!z)(2116\!+\!3093z\!+\!958z^2\!+\!62z^3)g^6
\!+\!\ldots \\
\end{split}
\end{equation}
and consequently
\begin{equation}
\tilde{G}_{2,2,2} = 2 z g^3\!+\!3 z(4\!+\!3z) g^4\!+\!6z(9\!+\!16z\!+\!4z^2)g^5\!+\!z(220\!+\!667z\!+\!399 z^2\!+\!50 z^3) g^6\!+\!\ldots\\
\end{equation}
whose first terms may be checked by a simple inspection.

As a final exercise, let us look at the $z\to 0$ limit of $\tilde{G}_{d_{12},d_{13},d_{23}}$. Letting $\alpha \to 0$ in \eqref{eq:TiUitilde},
we have
\begin{equation}
g=\frac{x}{(1+x)^2}+O(\alpha)\ , \qquad z=(1-x)^3(1+x) \alpha +O(\alpha^2)\ .
\end{equation} 
and we find for $\tilde{F}_{s,t,u}$ the same leading term (linear in $z$) as we found in the previous section for $F_{s,t,u}^{\text{even}}$,
so that again
\begin{equation}
\begin{split}
& \tilde{G}_{d_{12},d_{13},d_{23}} =2\, x^{s+t+u}\, z+O(z^2)=2\, x^{\frac{d_{12}+d_{13}+d_{23}}{2} }\, z+O(z^2)\\
&\text{with}\ \  x= g\, \left( \frac{1-\sqrt{1-4g}}{2g}\right)^2\ .\\
\label{eq:threepointzzerobip}
\end{split}
\end{equation} 
We recover the three-point function of planar trees, which is of course not a surprise since trees are automatically bipartite. 

\section{Applications}
\label{sec:applications}

\subsection{Critical line}
\label{sec:critline}
Throughout this section, we will enumerate maps {\it with a fixed value of $z$}, ranging from $0$ to $\infty$. 
The limit of maps with a large number of edges may then be captured by looking at the singularities in the variable $g$
of the various generating functions at hand. More precisely, these generating functions become singular when $g$ reaches a {\it critical value}
$g_{\text{crit}}\equiv g_{\text{crit}}(z)$ depending on $z$. The points $(z,g_{\text{crit}}(z))$ define the 
so-called {\it critical line} in the $(z,g)$ plane. In the case of general planar maps, this critical line may be found by looking for instance at the
singularities of the generating functions $T$ and $U$ introduced in \eqref{eq:TiUi}. These functions satisfy
the following recursion relations (which follow directly from the definition of $T_s$ and $U_s$ as generating functions for $(s)^+$-well-labelled
maps, which are particular instances of trees)
\begin{equation}
E_1\equiv T-z-g\, (2 U\, T+T^2)=0\ , \qquad E_2\equiv U-1-g\, (2 U\, T+U^2)=0\ .
\end{equation}
The location of their singularities is obtained by writing
\begin{equation}
0=\det\left( 
\begin{matrix}
\partial_T E_1 & \partial_T E_2 \\
\partial_U E_1 & \partial_U E_2 \\
\end{matrix}
\right)=4 g^2\, U^2+4 g^2\, T^2+4 g^2\, U\, T-4 g\, U-4 g\, T +1\ ,
\end{equation}
which upon setting $U=\upsilon/g$ and $T=\tau/g$ reads
\begin{equation}
4 \upsilon^2+4\tau^2 +4 \upsilon\tau -4 \upsilon -4 \tau +1=0\ .
\label{eq:ellipse}
\end{equation}
The values of $g_{\text{crit}}$ and $z$ are then recovered by writing $E_1=E_2=0$, namely
\begin{equation}
g_{\text{crit}}=\upsilon(1\!-\!\upsilon\!-\!2\tau)\ , \qquad z= \frac{\tau  (1\!-\!\tau\!-\!2 \upsilon)}{\upsilon  (1\!-\!\upsilon\!-\!2 \tau)}\ .
\end{equation}
The ellipse \eqref{eq:ellipse} may be parametrized as
\begin{equation}
\upsilon= \frac{2 r^2}{3 \left(r^2+3\right)}\ , \qquad \tau=\frac{(r-3)^2}{6 \left(r^2+3\right)}\ ,
\end{equation}
which yields the following parametrization of the critical line
\begin{equation}
z= \frac{(3-r)^3 (r+1)}{16 r^3}\, \qquad g_{\text{crit}}(z)=\frac{4 r^3}{3 \left(r^2+3\right)^2} \ ,
\label{eq:paramcrit}
\end{equation}
with $r$ varying from $3$ down to $0$ when $z$ varies from $0$ to $\infty$ (and in particular $r=1$ for $z=1$). Note that 
going from $z$ to $1/z$ corresponds in the parameter $r$ to performing the involution $r\to (3-r)/(1+r)$, and that 
$g_{\text{crit}}(1/z)=z\, g_{\text{crit}}(z)$. This property is a direct consequence of the trivial bijection which associates to each map 
its dual map. For $z\to 0$, we find $g_{\text{crit}}(0)=1/4$, consistent with a number of trees with $n$ edges growing like $4^n$,
while for $z=1$, we find $g_{\text{crit}}(1)=1/12$, consistent with a number of planar maps with $n$ edges 
growing like $12^n$. Finally, for $z\to \infty$, we find $g_{\text{crit}}(z)\sim1/(4z)$ as expected by duality with the $z\to 0$ limit
(the dominant configurations at large $z$ are dual to trees).

\subsection{Scaling limit}
\label{sec:scaling}
All the generating functions for general maps introduced in this paper are singular, for a fixed $z$, when $g\to g_{\text{crit}}(z)$.
The so-called {\it scaling limit} is then obtained by letting, for a fixed $z$, the weight $g$ tend to $g_{\text{crit}}(z)$ as
\begin{equation}
g=g_{\text{crit}}(z)(1-\epsilon^4)\ , \qquad \epsilon \to 0
\label{eq:gtogcrit}
\end{equation}
and letting {\it simultaneously all the distances} between the marked vertices {\it tend to} $\infty$ as $1/\epsilon$.
In other word, it amounts to write, in addition to \eqref{eq:gtogcrit},
\begin{equation}
d_{12}=\frac{\mathcal{D}_{12}}{\epsilon}\, \quad d_{13}=\frac{\mathcal{D}_{13}}{\epsilon}\, \quad d_{23}=\frac{\mathcal{D}_{23}}{\epsilon}\, \qquad
s=\frac{\mathcal{S}}{\epsilon}\, \quad t=\frac{\mathcal{T}}{\epsilon}\, \quad u=\frac{\mathcal{U}}{\epsilon}\
\label{eq:scaldist}
\end{equation}
with $\mathcal{D}_{12}$, $\mathcal{D}_{13}$, $\mathcal{D}_{23}$, $\mathcal{S}$, $\mathcal{T}$ and $\mathcal{U}$ remaining finite
when $\epsilon \to 0$.
The leading order in $\epsilon$ of the various generating functions then defines what we shall call {\it continuous scaling functions}.
The computation of these continuous scaling functions is in principle only a first step in getting the asymptotics of large maps. 
Some extra step is indeed required to extract from these functions properly normalized {\it continuous
canonical scaling functions} corresponding now to {\it genuine probability densities} for renormalized distances $\mathcal{D}_{ij}=d_{ij}/n^{1/4}$ 
in an ensemble of maps with 
{\it a fixed number $n$ of edges}, in the limit $n\to \infty$. The reader is invited to consult \cite{GEOD} for instance for an explicit example
of how to perform this second step. Still, as we shall now see, a number of large $n$ asymptotic results are directly readable 
from the continuous scaling functions themselves.

From \eqref{eq:TiUi}, fixing simultaneously the parameters $z$ as in \eqref{eq:paramcrit} and $g$ as in \eqref{eq:gtogcrit} is achieved by adjusting simultaneously $x$ and $\alpha$ as functions of $r$ and $\epsilon$. When $\epsilon\to 0$, both $x$ and $\alpha$ tend to $1$ and we
find in particular 
\begin{equation}
x=1- 2\, \gamma \, \epsilon +O(\epsilon^2)\ , \qquad \gamma= \sqrt{\frac{3(3-r) \sqrt{r^2+3}}{2r (r+3)}}\ .
\label{eq:gammaval}
\end{equation}
The quantity $\gamma$ will be referred to as the {\it scaling factor} in the following, as it fixes the scale for distances 
in the various continuous generating functions that we shall encounter. Its value at $z=1$ ($r=1$)
is $\sqrt{3/2}$.
Letting $\epsilon\to 0$, we find the expansions (these expansions require in practice expanding both $x$ and $\alpha$ at a sufficiently large
order in $\epsilon$, we skip the details here)
\begin{equation}
\begin{split}
& \hskip -10.pt N_{s,t} = \frac{6}{3+r} \Big\{(1\!+\!\epsilon \chi(\mathcal{S},\mathcal{T})\Big\}\!+\!O(\epsilon^2)\ , 
\frac{1}{1-g\, D_{s,t}} = \frac{3+r}{2r} \Big\{1\!+\!\epsilon \chi(\mathcal{S},\mathcal{T})\Big\}\!+\!O(\epsilon^2)\\ 
&\hskip -10.pt  \text{where}\ \chi\ \text{satisfies}\ \partial_{\mathcal{S}}\partial_{\mathcal{T}}\chi(\mathcal{S},\mathcal{T})=
2 \gamma^3 \frac{\cosh(\gamma(\mathcal{S}+\mathcal{T}))}{\sinh^3(\gamma(\mathcal{S}+\mathcal{T}))}\ .\\
\end{split}
\label{eq:NDscal}
\end{equation}
Both expressions lead, via \eqref{eq:twopointbisbivariate} and \eqref{eq:twopointterbivariate}, to
\begin{equation}
G_{d_{12}}=\epsilon^3 \mathcal{G}(\mathcal{D}_{12})\!+\!O(\epsilon^4)\ , \qquad  \mathcal{G}(\mathcal{D}_{12})=
2 \gamma^3 \frac{\cosh(\gamma\mathcal{D}_{12})}{\sinh^3(\gamma\mathcal{D}_{12})}\ .
\end{equation}
Here we recognize the {\it continuous two-point function} $\mathcal{G}(\mathcal{D}_{12})$ found in \cite{GEOD} for quadrangulations, up to
the $z$-dependent scaling factor $\gamma$, which appears both in the argument of the function, hence fixes the scale
for $\mathcal{D}_{12}$, and in its normalization. Note that at $r=1$ ($\gamma=\sqrt{3/2}$), this normalization 
differs by a factor of $2$ from that of quadrangulations. This is because there are asymptotically twice as many bi-pointed quadrangulations
with $n$ faces as bi-pointed planar maps with $n$ edges. More generally, the normalization is simply inherited from the asymptotics
of bi-pointed maps in the fixed $z$ ensemble. It is wiped out when extracting from $\mathcal{G}(\mathcal{D}_{12})$ the 
continuous canonical two-point function (the probability density for $\mathcal{D}_{12}=d_{12}/n^{1/4}$ for maps with a fixed, large $n$),
which thus differs from that of quadrangulations found in \cite{GEOD} only by the change of scale $\mathcal{D}_{12}\to \gamma \mathcal{D}_{12}$.

From \eqref{eq:NDscal}, we also deduce
\begin{equation}
\begin{split}
& \hskip -10.pt \Delta_s\Delta_t N_{s,t} \!=\! \epsilon^3\, \frac{6}{3\!+\!r} \mathcal{G}(\mathcal{D}_{12})+\!O(\epsilon^4)\ ,\quad 
\Delta_s\Delta_t  \frac{1}{1-g\, D_{s,t}}\!=\!\epsilon^3\, \frac{3\!+\!r}{2r} \mathcal{G}(\mathcal{D}_{12})+\!O(\epsilon^4)\ ,\\ 
&\hskip 100.pt  \text{with}\ \mathcal{S}+\mathcal{T}=\mathcal{D}_{12}\ .\\
\end{split}
\end{equation}
Without any further calculation, we immediately deduce from these equations that 
$\Delta_s\Delta_t N_{s,t}$ and $\Delta_s\Delta_t  1/(1-g\, D_{s,t})$ have,
up to an explicit normalization factor, the same large $n$ asymptotics as the two-point function $G_{d_{12}}$. From Remark 2,
the reader will realize that, as bi-pointed map generating functions, $\Delta_s\Delta_t N_{s,t}$  with $s+t=d_{12}$ 
(resp. $\Delta_s\Delta_t  1/(1-g\, D_{s,t})$ with
$s+t-1=d_{12}$) differs from the two-point 
function $G_{d_{12}}$ only by the marking of an extra vertex (resp. an extra edge) {\it belonging to a geodesic path} 
between the two marked vertices, this extra vertex being at distance $s$ from the first vertex (resp.\ this extra edge 
being the $s$-th edge along the path). 
We may therefore interpret
\begin{equation}
N_{\text{geod vertices}}=  \frac{6}{3+r} \quad \text{and}\quad N_{\text{geod edges}}=\frac{3+r}{2r}
\end{equation} 
as the average numbers of {\it geodesic vertices} and {\it geodesic edges} between two far-away (in practice at a distance of
order $n^{1/4}$) vertices, at a fixed distance (itself of order $n^{1/4}$) from the first vertex, in the ensemble of bi-pointed general
planar maps with $n$ edges, in the limit $n\to \infty$. Note that these numbers do not depend on the position along the geodesic path
in this limit.
For $z=1$ ($r=1$), we find $N_{\text{geod vertices}}=3/2$ and
$N_{\text{geod edges}}=2$. For $z\to 0$ ($r\to 3$), we find $N_{\text{geod vertices}}=N_{\text{geod edges}}=1$,
in agreement with the fact that the map then degenerates into a tree, with a single geodesic
path between two given vertices. 

As for the three-point function, we find that both $F^{\text{even}}_{s,t,u}$ and $F^{\text{odd}}_{s,t,u}$ behave at leading order in
$\epsilon$ as $\epsilon^{-2} {\mathcal F}(\mathcal{S},\mathcal{T},\mathcal{U})$ where
\begin{equation}
\hskip -5.pt 
{\mathcal F}(\mathcal{S},\mathcal{T},\mathcal{U}) =
\frac{3(3\!-\!r)^2}{2(3\!+\!r)^3\gamma^2}
\left(\frac{2\sinh(\gamma(\mathcal{S}\!+\!\mathcal{T}\!+\!\mathcal{U})) \sinh(\gamma\, \mathcal{S}) \sinh(\gamma\, \mathcal{T}) \sinh(\gamma \,\mathcal{U})
}{ \sinh(\gamma(\mathcal{S}\!+\!\mathcal{T})) \sinh(\gamma(\mathcal{T}\!+\!\mathcal{U})) \sinh(\gamma(\mathcal{U}\!+\!\mathcal{S}))
}\right)^2
\label{eq:Fscal}
\end{equation}
so that
\begin{equation}
\begin{split}
& \hskip -10pt G_{d_{12},d_{13},d_{23}}=  {\mathcal G}(\mathcal{D}_{12},\mathcal{D}_{13},\mathcal{D}_{23})\,  \epsilon\!+\! O(\epsilon^2)
\quad \text{with}\ {\mathcal G}=\partial_{\mathcal{S}}\partial_{\mathcal{T}}\partial_{\mathcal{U}}{\mathcal F}(\mathcal{S},\mathcal{T},\mathcal{U})\\
& \\
& \hskip -10pt 
\text{and}\  \mathcal{D}_{12}=\mathcal{S}+\mathcal{T}\ ,  \mathcal{D}_{13}=\mathcal{S}+\mathcal{U}\ ,  \mathcal{D}_{23}=\mathcal{T}+\mathcal{U}\ .\\
\end{split}
\end{equation}
Here again we recognize the continuous three-point function ${\mathcal G}(\mathcal{D}_{12},\mathcal{D}_{13},\mathcal{D}_{23})$ found in \cite{BG08} for quadrangulations, now however with a new scale for the distances $\mathcal{D}_{ij}$, fixed by the $z$-dependent 
scaling factor $\gamma$, and with a new normalization. This 
normalization is inherited from the asymptotics of tri-pointed maps and we may understand it as follows: 
when $s$, $t$ and $u$ tend to $\infty$, $F^{\text{even}}_{s,t,u}+F^{\text{odd}}_{s,t,u}$ tends to the generating 
function of tri-pointed planar maps, in the fixed $z$ ensemble, with no constraint on the pairwise distances between the three marked vertices. 
Letting  accordingly $\mathcal{S}$, $\mathcal{T}$ and $\mathcal{U}$
tend to $\infty$ in $2{\mathcal F}(\mathcal{S},\mathcal{T},\mathcal{U})$ (in which case the term within the big parentheses in \eqref{eq:Fscal} 
tends to $1$), the normalization in front measures the $\epsilon^{-2}=1/\sqrt{1-g/g_{\text{crit}}}$ singularity of the generating 
function of tri-pointed planar maps, hence leads to the following asymptotics, at large number $n$ of edges, of the ``number"
of tri-pointed maps (enumerated with a weight $z$ per face)
\begin{equation}
\#\text{tri-pointed maps}\ \sim \frac{3(3-r)^2}{(3+r)^3\gamma^2} \frac{g_{\text{crit}}^{-n}}{\sqrt{\pi} n^{1/2}}\ .
\end{equation} 
Understanding the normalization in \eqref{eq:Fscal} is therefore equivalent to understanding this asymptotics.

Now, from the expansions (see \cite{BFG} for a precise meaning of these functions as map generating functions)
\begin{equation}
\begin{split}
& \log(1+g\, U\, T)=\text{const.}-\frac{\gamma^2}{3}\epsilon^2+O(\epsilon^3)\ , \\
&1+2 g\, U\, T +g\, T^2 = \text{const.}-\frac{(3+r)(3+r^2)\gamma^2}{12r^2}\epsilon^2+O(\epsilon^3)\ , \\
\end{split}
\end{equation}
we read the asymptotics of the numbers of bi-pointed and pointed-rooted (with a marked vertex and a marked oriented edge)
planar maps at large number $n$ of edges
\begin{equation}
\begin{split}
&\#\text{bi-pointed maps}\ \sim \frac{\gamma^2}{3}  \frac{g_{\text{crit}}^{-n}}{2\sqrt{\pi} n^{3/2}}\ ,\\
&\#\text{pointed-rooted maps}\ \sim \frac{(3+r)(3+r^2)\gamma^2}{12r^2} \frac{g_{\text{crit}}^{-n}}{2\sqrt{\pi} n^{3/2}}\ .\\
\end{split}
\end{equation}
Both could have been obtained by starting from simply pointed maps and choosing an extra vertex in the first case or  
an extra oriented edge in the second case.
Since there are $2n$ choices of oriented edges, we deduce by comparison that, in an ensemble of planar maps with a fixed number of edges $n$, 
and with a fixed weight $z$ per face, the average number $n_v\equiv n_v(z)$ of vertices and that, $n_f\equiv n_f(v)=n+2-n_v$, of faces are,
asymptotically at large $n$, given by
\begin{equation}
n_v\sim \frac{8r^2}{(3+r)(3+r^2)} \ n\ , \qquad n_f\sim \frac{(1+r)(3-r)^2}{(3+r)(3+r^2)} \ n\
\end{equation}
(these asymptotics should not depend on the precise ensemble, namely whether the maps are simply pointed,
as in the above argument, multiply pointed, rooted or not). 
For $z=1$, ($r=1$), $n_v\sim n_f\sim n/2$ as expected since by duality, there are on average as many faces 
as vertices. More generally, the expressions for $n_v$ and $n_f$ are exchanged under $r\to (3-r)/(1+r)$ (corresponding
to $z\to 1/z$) as a consequence of duality. For $z\to 0 $ ($r\to 3$), 
$n_v\sim n$ as expected for trees (or more generally maps with a finite number of faces) which have asymptotically as many vertices as edges.

Returning to tri-pointed maps, they may be obtained similarly from bi-pointed maps by choosing
an extra third vertex. This explains eventually their asymptotics via the identity
\begin{equation}
\frac{8r^2}{(3+r)(3+r^2)} \, n \times \frac{\gamma^2}{3}  \frac{g_{\text{crit}}^{-n}}{2\sqrt{\pi} n^{3/2}}\ =  \frac{3(3-r)^2}{(3+r)^3\gamma^2} 
\frac{g_{\text{crit}}^{-n}}{\sqrt{\pi} n^{1/2}}\ .
\end{equation}

To conclude, if we now extract from \eqref{eq:Fscal} the continuous canonical
three-point function (corresponding, in an ensemble of maps with a large fixed number $n$ of edges
and with three marked vertices, to the probability density for prescribed pairwise renormalized distances $\mathcal{D}_{12}=d_{12}/n^{1/4}$, 
$\mathcal{D}_{13}=d_{13}/n^{1/4}$
and $\mathcal{D}_{23}=d_{23}/n^{1/4}$), all prefactors are wiped out by normalization so that this canonical continuous three-point function is exactly the same as  that found in \cite{BG08} for quadrangulations
{\it apart from a global change of scale of the renormalized distances} $\mathcal{D} \to \gamma \mathcal{D}$.  In particular,
its $z$-dependence is entirely contained in the scaling factor $\gamma$ via \eqref{eq:gammaval}.

\subsection{Bipartite maps}
\label{sec:biparapp} 
Let us repeat our derivation of the critical line in the $(z,g)$ plane, now in the case of bipartite planar maps. To this end, we
now look at the singularities of the generating functions $\tilde{T}$ and $\tilde{U}$ introduced in \eqref{eq:TiUitilde}. These functions satisfy
the recursion relations
\begin{equation}
\tilde{E}_1\equiv \tilde{T}-z-g\, 2 \tilde{U}\, \tilde{T}=0\ , \qquad \tilde{E}_2\equiv \tilde{U}-1-g\, (\tilde{U}\, \tilde{T}+\tilde{U}^2)=0
\end{equation}
and the location of their singularities is obtained from
\begin{equation}
0=\det\left( 
\begin{matrix}
\partial_{\tilde{T}} \tilde{E}_1 & \partial_{\tilde{T}} \tilde{E}_2 \\
\partial_{\tilde{U}} \tilde{E}_1 & \partial_{\tilde{U}} \tilde{E}_2 \\
\end{matrix}
\right)=4 g^2\, U^2-4 g\, U- g\, T +1\ ,
\end{equation}
which upon setting $\tilde{U}=\tilde{\upsilon}/g$ and $\tilde{T}=\tilde{\tau}/g$ reads
\begin{equation}
4 \tilde{\upsilon}^2-4 \tilde{\upsilon} -\tilde{\tau} +1=0\ .
\label{eq:parabola}
\end{equation}
Writing $\tilde{E}_1=\tilde{E}_2=0$, the values of $g_{\text{crit}}$ and $z$ follow
\begin{equation}
 z=\frac{(1-2 \tilde{\upsilon} )^3}{(3-4 \tilde{\upsilon} ) \tilde{\upsilon}^2}\ , \qquad
g_{\text{crit}}=(3-4 \tilde{\upsilon} ) \tilde{\upsilon}^2\ ,
\label{eq:critlinebip}
\end{equation}
with $\tilde{\upsilon}$ varying from $1/2$ down to $0$ when $z$ varies from $0$ to $\infty$ (and in particular $\tilde{\upsilon}=1/4$ for $z=1$). 
For $z\to 0$, we find again $g_{\text{crit}}(0)=1/4$, consistent with the fact that we enumerate trees,
while for $z=1$, we find $g_{\text{crit}}(1)=1/8$, consistent with a number of planar bipartite maps with $n$ edges 
growing like $8^n$. Finally, for $z\to \infty$, we find $g_{\text{crit}}(z)\sim1/z$. Note that there is no duality symmetry for bipartite maps. 

The scaling limit is again reached, for a fixed value of $z$, by letting $g$ tend to its critical value \eqref{eq:critlinebip} exactly as in \eqref{eq:gtogcrit}
and letting simultaneously all the distances between the marked vertices tend to $\infty$ as in \eqref{eq:scaldist}.
From \eqref{eq:TiUitilde}, this amounts to again adjust $x$ and $\alpha$ as functions of $\tilde{\upsilon}$ and $\epsilon$ and,
when $\epsilon\to 0$, we now find  
\begin{equation}
x=1- 2\, \gamma\, \epsilon +O(\epsilon^2)\ , \qquad \gamma= \sqrt{\frac{\sqrt{3}(1-2\tilde{\upsilon})}{2\sqrt{\tilde{\upsilon}(1-\tilde{\upsilon})}}}\ .
\label{eq:gammatildeval}
\end{equation}

The value of the scaling factor at $z=1$ ($\tilde{\upsilon}=1/4$)
is now $\gamma=1$.
Letting $\epsilon\to 0$, we find the expansion
\begin{equation}
\begin{split}
&  \tilde{X}_{s,t} = (3-4\tilde{\upsilon}) \Big\{(1\!+\!\epsilon \tilde{\chi}(\mathcal{S},\mathcal{T})\Big\}\!+\!O(\epsilon^2)\\
& \text{where}\ \tilde{\chi}\ \text{satisfies}\ \partial_{\mathcal{S}}\partial_{\mathcal{T}}\tilde{\chi}(\mathcal{S},\mathcal{T})=
4 \gamma^3 \frac{\cosh(\gamma(\mathcal{S}+\mathcal{T}))}{\sinh^3(\gamma(\mathcal{S}+\mathcal{T}))}\\
\end{split}
\label{eq:Xscalbip}
\end{equation}
and consequently
\begin{equation}
\tilde{G}_{d_{12}}=\epsilon^3 \tilde{\mathcal{G}}(\mathcal{D}_{12})\!+\!O(\epsilon^4)\ , \qquad  \tilde{\mathcal{G}}(\mathcal{D}_{12})=
4 \gamma^3 \frac{\cosh(\gamma\mathcal{D}_{12})}{\sinh^3(\gamma\mathcal{D}_{12})}\ ,
\end{equation}
where we recognize again the continuous two-point function of \cite{GEOD}, up to the 
$z$-dependent scaling factor $\gamma$. 

Looking at the prefactor in \eqref{eq:Xscalbip}, we deduce also the average number of geodesic vertices between two far-away vertices,
at some arbitrary but fixed distance from the first vertex
\begin{equation}
N_{\text{geod vertices}}= 3-4\tilde{\upsilon}\ .
\end{equation} 
The formula gives $N_{\text{geod vertices}}=2$ for $z=1$ ($\tilde{\upsilon}=1/4$) and $N_{\text{geod vertices}}=1$ for $z\to 0$ ($\tilde{\upsilon}\to1/2$),
as expected for trees. 

From the expansions (see \cite{BFG} for the interpretation of these functions as bipartite map generating functions)
\begin{equation}
\begin{split}
& \log(1+g\, \tilde{U}\, \tilde{T})=\text{const.}-\frac{2\gamma^2}{3}\epsilon^2+O(\epsilon^3)\ , \\
&1+2 g\, \tilde{U}\, \tilde{T}  = \text{const.}-\frac{4(1-\tilde{\upsilon})\gamma^2}{3\tilde{\upsilon} (3-4\tilde{\upsilon})}\epsilon^2+O(\epsilon^3)\ , \\
\end{split}
\end{equation}
we read the asymptotics of the numbers of bi-pointed and pointed-rooted bipartite planar maps at large number $n$ of edges
\begin{equation}
\begin{split}
&\#\text{bi-pointed bipartite maps}\ \sim \frac{2\gamma^2}{3}\  \frac{g_{\text{crit}}^{-n}}{2\sqrt{\pi} n^{3/2}}\ ,\\
&\#\text{pointed-rooted bipartite maps}\ \sim \frac{4(1-\tilde{\upsilon})\gamma^2}{3\tilde{\upsilon} (3-4\tilde{\upsilon})}\frac{g_{\text{crit}}^{-n}}{2\sqrt{\pi} n^{3/2}}\ ,\\
\end{split}
\end{equation}
with $g_{\text{crit}}$ as in \eqref{eq:critlinebip}
and, by comparison, the asymptotics of the average numbers $n_v=n_v(z)$ of vertices and $n_f=n_f(z)$ of faces 
\begin{equation}
n_v\sim \tilde{\upsilon} \frac{3-4 \tilde{\upsilon}}{1-\tilde{\upsilon}} \ n\ , \qquad n_f\sim \frac{(1-2 \tilde{\upsilon})^2}{1-\tilde{\upsilon}} \ n\ .
\end{equation}
For $z=1$ ($\tilde{\upsilon}=1/4$), $n_v\sim 2n/3$ and $n_f\sim n/3$, as expected since, via a trivial bijection with Eulerian triangulations,
bipartite maps have equally distributed numbers of faces, of vertices of even color and of vertices of odd color.
For $z\to 0 $ ($\tilde{\upsilon}\to1/2$), $n_v\sim n$, as expected.

Using
\begin{equation}
\tilde{\upsilon} \frac{3-4 \tilde{\upsilon}}{1-\tilde{\upsilon}} \ n \times \frac{2\gamma^2}{3}\  \frac{g_{\text{crit}}^{-n}}{2\sqrt{\pi} n^{3/2}}=  
\frac{(1-2 \tilde{\upsilon} )^2 (3-4 \tilde{\upsilon})}{4 (1-\tilde{\upsilon} )^2\gamma^2} \frac{g_{\text{crit}}^{-n}}{\sqrt{\pi} n^{1/2}}\ ,
\end{equation}
for $\gamma$ as in \eqref{eq:gammatildeval}, the asymptotics of tri-pointed maps is then
\begin{equation}
\#\text{tri-pointed bipartite maps}\ \sim  \frac{(1-2 \tilde{\upsilon} )^2 (3-4 \tilde{\upsilon})}{4 (1-\tilde{\upsilon} )^2\gamma^2}  \frac{g_{\text{crit}}^{-n}}{\sqrt{\pi} n^{1/2}}\ .
\end{equation} 
This is consistent with the expansion $\tilde{F}_{s,t,u}\sim \epsilon^{-2} \tilde{{\mathcal F}}(\mathcal{S},\mathcal{T},\mathcal{U})$ for which
we find 
\begin{equation}
\hskip -2.pt 
\tilde{{\mathcal F}}(\mathcal{S},\mathcal{T},\mathcal{U}) =
\frac{(1\!-\!2 \tilde{\upsilon} )^2 (3\!-\!4 \tilde{\upsilon})}{4 (1\!-\!\tilde{\upsilon} )^2\gamma^2}
\left(\frac{2\sinh(\gamma(\mathcal{S}\!+\!\mathcal{T}\!+\!\mathcal{U})) \sinh(\gamma\, \mathcal{S}) \sinh(\gamma\, \mathcal{T}) \sinh(\gamma \,\mathcal{U})
}{ \sinh(\gamma(\mathcal{S}\!+\!\mathcal{T})) \sinh(\gamma(\mathcal{T}\!+\!\mathcal{U})) \sinh(\gamma(\mathcal{U}\!+\!\mathcal{S}))
}\right)^2\ .
\label{eq:Ftildescal}
\end{equation}
Note that the continuous three-point function is now
\begin{equation}
\begin{split}
& \hskip 30pt \tilde{{\mathcal G}}(\mathcal{D}_{12},\mathcal{D}_{13},\mathcal{D}_{23})=\frac{1}{2}\partial_{\mathcal{S}}\partial_{\mathcal{T}}\partial_{\mathcal{U}}\tilde{{\mathcal F}}(\mathcal{S},\mathcal{T},\mathcal{U})\\
& \\
& \hskip 30pt 
\text{where}\  \mathcal{D}_{12}=\mathcal{S}+\mathcal{T}\ ,  \mathcal{D}_{13}=\mathcal{S}+\mathcal{U}\ ,  \mathcal{D}_{23}=\mathcal{T}+\mathcal{U}\ ,\\
\end{split}
\end{equation}
with a $1/2$ factor arising from the parity constraint on the sum of the three pairwise distances.
Again, the (normalized) continuous canonical
three-point function is exactly the same as  that for quadrangulations or general maps
except for the global change of scale of the renormalized distances through the scaling factor 
$\gamma$ of \eqref{eq:gammatildeval}.

\section{Conclusion}
\label{sec:conclusion}

In this paper, we established several bijections between bi-pointed or tri-pointed, general or bipartite, planar maps
with prescribed distances between their marked vertices and suitably defined well-labelled or 
very-well-labelled objects. From the generating functions of these latter objects, we were able to
recover known formulas for the two-point function of maps and to derive new explicit
expressions for their three-point function, with a control on both their numbers of edges and faces.

The techniques that we used are in all respects similar to those 
used in \cite{BG08} where the same questions were addressed in the context of planar quadrangulations. 
The intermediate well-labelled building blocks are essentially the same, with only slight differences in the game rules:
for instance, while the three-point function of quadrangulations involves $(s,t,u)^+$-well-labelled maps enumerated by
$X_{s,t}X_{s,u}X_{t,u}Y_{s,t,u}^2$, that of general maps is restricted to the subclass of those maps being of type A,
enumerated by $N_{s,t}N_{s,u}N_{t,u}Y_{s,t,u}^2$ or of type B, enumerated by $O_{s,t}O_{s,u}O_{t,u}Y_{s,t,u}^2$.

These small differences, although crucial at the discrete level, turn out to disappear in the scaling limit of large
maps. In particular, the continuous canonical two- and three-point functions of general planar maps with a large number
$n$ of edges are {\it exactly the same} as those of planar quadrangulations with a large number $n$ of faces.
As a consequence, all the different limiting situations discussed in \cite{BG08} are recovered for general maps.

As for general or bipartite planar maps enumerated with a fixed value of the weight $z$ per face, 
their continuous canonical two- and three-point functions at large number $n$ of edges are again the same
as for quadrangulations up to a global, $z$-dependent, change of scale in the distances by 
the scaling factor $\gamma$ of \eqref{eq:gammaval} or \eqref{eq:gammatildeval}. In this sense,
these functions are {\it universal} and this corroborates the belief that all the families of maps considered here 
lie in the same universality class, that of so-called ``pure gravity". In another, more probabilistic direction,
this universality is also confirmed, for $z=1$, by the fact that the so-called Brownian map \cite{LEGALL07} is the universal 
scaling limit of general and bipartite maps counted by their number of edges \cite{BJM13,CA14}.

Note that the universality class of pure gravity is reached provided $\gamma$ is finite and non-zero, which therefore excludes the cases $z\to 0$
($\gamma\to 0$) and $z\to \infty$ ($\gamma\to \infty$), which are thus in a different universality class.
For $z\to 0$, one finds (omitting prefactors) that, for both general and bipartite maps, $\gamma\simeq z^{1/6}$
and $n_f/n \simeq z^{2/3}$, and therefore $\gamma\simeq (n_f/n)^{1/4}$. This suggests that, when dealing
with an ensemble where we impose that $n_f\simeq n^b$ for some $b$, with $0\leq b<1$, the statistics of distances 
should be describable by an effective $n$-dependent scaling factor $\gamma$ satisfying $\gamma\simeq n^{(b-1)/4}$. 
This would indicate that the typical (discrete) distance $d$, of order $n^{1/4}/\gamma$, would scale as $d\simeq n^{(2-b)/4}$. 
For $b=0$ (i.e., with a finite number of faces), this gives $d\simeq n^{1/2}$ as expected. For $b\to 1$, the exponent 
tends to $1/4$ as expected. For $z\to \infty$ instead, one finds for general maps $\gamma\simeq z^{1/6}$
and $n_v/n \simeq z^{-2/3}$, and for bipartite maps $\gamma\simeq z^{1/8}$
and $n_v/n \simeq z^{-1/2}$. Therefore $\gamma\simeq (n_v/n)^{-1/4}$ in both cases. Hence, in an ensemble where 
we impose that $n_v\simeq n^c$ for some $c$, with $0\leq c<1$, we expect that the statistics of distances should be 
governed by an effective $n$-dependent scaling
factor $\gamma\simeq n^{(1-c)/4}$. This would indicate that the typical (discrete) distance $d$, of order $n^{1/4}/\gamma$, 
would now scale as $d\simeq n^{c/4}$. 
For $c=0$ (i.e., with a finite number of vertices), this gives $d\simeq n^{0}$ as expected, since distances
remain finite in maps with a finite number of vertices. For $c\to 1$, the exponent tends to $1/4$ as expected.
All these conjectures require careful calculations to be confirmed.

Even in the scaling limit, some non-universal quantities (other than the scaling factor) can be read off the normalizations of the various
scaling functions: we discussed for instance the asymptotic average numbers of vertices and faces in maps
with a fixed weight $z$ per face, as well as the number of geodesic vertices and edges between two points
far away. Many other non-universal quantities follow from our formulas: in particular, we could also easily
explore the (non-universal) so-called local limit, corresponding to vertices at a finite discrete distance in large (potentially infinite) maps. 

To conclude, let us mention that, in the context of quadrangulations, several refinements  of \cite{BG08} led to a more precise description of the
``geodesic triangle" formed by three random points (with in particular a measure of the so-called ``confluence phenomenon"
for geodesics) as well as to an evaluation of the distance properties of ``separating loops" \cite{BG09}.
There should not be any difficulty to apply the same refinements to our calculations and address similar questions
now in the context of general or bipartite planar maps. 

\section*{Acknowledgements} The work of
\'EF was partly supported by the ANR grant  
``Cartaplus'' 12-JS02-001-01 and the ANR grant ``EGOS'' 12-JS02-002-01.

\bibliographystyle{plain}
\bibliography{gen3p}

\end{document}